%% file: ms.tex
\documentclass{amsart}
\usepackage{amssymb}
\usepackage[all]{xy}
\UseComputerModernTips
\CompileMatrices
\usepackage{hyperref}
\newtheorem{theorem}{Theorem}[section]
\newtheorem{lemma}[theorem]{Lemma}
\newtheorem{proposition}[theorem]{Proposition}
\newtheorem{corollary}[theorem]{Corollary}
\theoremstyle{definition}
\newtheorem{definition}[theorem]{Definition}
\newtheorem{condition}[theorem]{Condition}

\theoremstyle{remark}

\input CSdefs.tex
\begin{document}
%
%topmatter
\title[Equivariant surgery for compact Lie groups]
{The equivariant Spivak normal bundle and equivariant
surgery for compact Lie groups}

\author{Steven R. Costenoble}
\address{Department of Mathematics\\103 Hofstra University\\
   Hempstead, NY 11550}
\email{Steven.R.Costenoble@Hofstra.edu}

\author{Stefan Waner}
%\address{Department of Mathematics\\103 Hofstra University\\
%   Hempstead, NY 11550}
\email{Stefan.Waner@Hofstra.edu}


\date{\today}

\abstract
We generalize the results of \cite{CW:spivaknormal}
to compact Lie groups. Using a recently developed ordinary equivariant homology and
cohomology, we define equivariant Poincar\'e complexes with the properties
that (1) every compact $G$-manifold is an equivariant Poincar\'e complex,
(2) every finite equivariant Poincar\'e complex 
(with some mild additional hypotheses) has an equivariant spherical Spivak
normal fibration, and
(3) the $\Pi$-$\Pi$ Theorem holds for equivariant Poincar\'e pairs under
suitable gap hypotheses.
The nice behavior of the ordinary equivariant homology and cohomology theories
allows us to follow Wall's original line of argument closely.
\endabstract

\maketitle
\tableofcontents
%
%mainmatter
\include{CSbody}

%backmatter
%\bibliographystyle{amsalpha}
\bibliographystyle{amsplain}
\bibliography{Topology}

\end{document}

%% file: CSdefs.tex
\DeclareFontFamily{U}{rsfs}{}
\DeclareFontShape{U}{rsfs}{m}{n}{%
   <5> <6> rsfs5
   <7> rsfs7
   <8> <9> <10> <10.95> <12> <14.4> <17.28> <20.74> <24.88> rsfs10
}{}
\DeclareSymbolFont{rsfs}{U}{rsfs}{m}{n}
\DeclareSymbolFontAlphabet{\mathscr}{rsfs}
\renewcommand{\phi}{\varphi}

\newcommand{\gpdact}{\odot}
\newcommand{\bndry}{\partial}
\DeclareMathSymbol{\boxprod}{\mathbin}{AMSa}{"03} % binary operator version of \square
\DeclareMathSymbol{\mixprod}{\mathbin}{AMSa}{"4F} % binary operator version of \triangledown
\newcommand{\congr}{\equiv}

\newcommand{\convto}{\Rightarrow}
\newcommand{\dirsum}{\oplus}
\newcommand{\Dirsum}{\bigoplus}

\newcommand{\from}{\leftarrow}

\newcommand{\hmtpc}{\simeq}
\newcommand{\homeo}{\approx}
\newcommand{\intersect}{\cap}
\newcommand{\includesin}{\hookrightarrow}
\newcommand{\iso}{\cong}
\newcommand{\laxto}{\Rightarrow}

\newcommand{\Mackey}[1]{\overline{#1}\vphantom{#1}}
\newcommand{\MackeyOp}[1]{{\underline {#1}}}
\newcommand{\smsh}{\wedge}

\newcommand{\susp}{\Sigma}
\newcommand{\tensor}{\otimes}
\newcommand{\union}{\cup}

\newcommand{\ParaQuot}[1]{/\!_{#1}\,}

\newcommand{\Cplx}{{\mathbb C}}

\newcommand{\DA}{{\mathbb D}}

\renewcommand{\H}{{\mathscr H}}
\newcommand{\K}{{\mathscr K}}
\newcommand{\Qtrn}{{\mathbb H}}

\newcommand{\Real}{\mathbb{R}}

\newcommand{\Z}{\mathbb{Z}}

\newcommand{\Lie}{{\mathscr L}}

\newcommand{\innprod}[3]{\langle {#1}, {#2} \rangle_{#3}}

\newcommand{\orb}[1]{{\mathscr{O}_{#1}}}

\newcommand{\stab}[1]{{\widehat{#1}}}
\newcommand{\sorb}[1]{{\stab{\mathscr O}_{#1}}}

\newcommand{\ortho}[2]{{\mathscr V}_{#1}^{#2}}

\newcommand{\ParaU}[1]{\K/{#1}}
\newcommand{\Para}[1]{\K_{#1}}

\newcommand{\Spec}[2]{{\mathscr S}#1_{#2}}

\newcommand{\Irr}{{\mathscr I}}

\DeclareMathOperator{\id}{id}

\DeclareMathOperator{\Aut}{Aut}
\DeclareMathOperator*{\colim}{colim}
\DeclareMathOperator{\End}{End}
\DeclareMathOperator{\Hom}{Hom}

\DeclareMathOperator{\Ind}{Ind}
\DeclareMathOperator{\Res}{Res}

\DeclareMathOperator{\Map}{Map}

\DeclareMathOperator{\Tor}{Tor}
\DeclareMathOperator{\Ext}{Ext}
\DeclareMathOperator{\Mod}{Mod}

%% file: CSbody.tex
% !TEX root = CompactSurgery.tex
\section*{Introduction}

Equivariant surgery was pioneered by Dovermann, Petrie, and Rothenberg about 40 years ago
in a series of papers including
\cite{Pe:projectiveClassGroup}, \cite{Pe:spheresII}, \cite{DP:Gsurgery}, and \cite{DR:surgery}.
They worked predominantly with finite group actions, with
some preliminary work on compact Lie group actions.
Moreover, they assumed that all fixed sets were simply connected, and they sought
results only up to pseudo-equivalence.
One of their results was a $\Pi$-$\Pi$ theorem under these assumptions.

We would like a $\Pi$-$\Pi$ theorem allowing nontrivial fundamental groups and
working with true $G$-equivalence in the full generality of compact Lie group actions.
For this, 
we need a good theory of Poincar\'e duality,
which means we need good equivariant homology and cohomology theories.
In particular, these theories need to exhibit Poincar\'e duality for all 
smooth compact $G$-manifolds,
and detect $G$-equivalences (to the same degree that nonequivariant homology detects equivalences).
Since such theories have been lacking, it is no surprise that
little progress has been made on a $\Pi$-$\Pi$ theorem of the generality we want.

When $G$ is finite, we defined the appropriate ordinary homology and cohomology theories
in \cite{CW:spivaknormal}.
These theories allowed us to define equivariant Poincar\'e complexes
for finite group actions and show that with this definition we could prove
three things:
(1) every compact $G$-manifold is an equivariant Poincar\'e duality
complex,
(2) every finite $G$-Poincar\'e duality complex has an equivariant
spherical Spivak normal fibration, and
(3) the $\Pi$-$\Pi$ theorem holds for $G$-Poincar\'e pairs under
suitable gap hypotheses.
(In fact, that paper was too optimistic with respect to (3);
it is probably not as easy as claimed to adapt the argument of
\cite{DP:Gsurgery} to the context there.)

Recently, in \cite{CW:homologybook}, we completed the construction of the ordinary homology
and cohomology theories needed for the compact Lie case.
With this machinery in place, we now carry out
surgery for compact Lie group actions, at least to the point of proving a
$\Pi$-$\Pi$ theorem.
However, we do not try to adapt the argument in \cite{DP:Gsurgery} to compact Lie group actions,
which would appear to be very difficult (those authors did not do so, either),
but rather take advantage of the power
of the ordinary homology and cohomology theories to use
arguments very similar to Wall's original arguments in \cite{Wal:surgery}.
(The work here also fills in any gaps in \cite{CW:spivaknormal}.)
%
%This paper owes a great deal to the earlier work of Dovermann, Petrie and Rothenberg.
%Although our context is more complicated in that we do not assume that all fixed
%sets are simply connected, it is simpler in that
%we are looking for $G$-equivalences rather than pseudo-equivalences, which makes
%arguments depending on induction on subgroups easier.
%Further, our well-behaved equivariant ordinary homology and cohomology theories with associated equivariant Poincar\'e duality give us the necessary tools to allow us to follow Wall's techniques in \cite{Wal:surgery} closely.
We content ourselves with getting a $\Pi$-$\Pi$ theorem and, although we do not here
push on to discuss obstructions to surgery in the general case,
we believe we have provided a framework to do so.

\section{Ordinary Equivariant Homology}

Nonequivariantly, Wall \cite{Wal:surgery} used ordinary homology with local coefficients
to deal with spaces with nontrivial fundamental groups.
Equivariantly, we use the ordinary equivariant homology theory developed in \cite{CW:homologybook}
to deal with spaces with nontrivial fundamental groupoids and also with the problem that
Poincar\'e duality generally fails in integer or $RO(G)$-graded theories.
In this section we review the definition of this ordinary theory and some of its properties.

As discussed in \cite{CW:homologybook}, ordinary equivariant homology and cohomology theories are conveniently thought off as defined on a category of parametrized spaces over a given fixed basespace $X$. (Note: All spaces will be assumed to be compactly generated weak Hausdorff spaces without further comment.) 

\begin{definition}
Let $X$ be a $G$-space.
\begin{enumerate}
\item
Let $G\ParaU X$ be the category of {\em $G$-spaces over $X$}:
Its objects are pairs $(Y,p)$ where $Y$ is a $G$-space
and $p\colon Y\to X$ is a $G$-map.
A map $(Y,p)\to (Z,q)$ is a $G$-map
$f\colon Y\to Z$ such that $q\circ f = p$, i.e., a $G$-map over $X$.
\item
Let $G\Para X$ be the category of {\em ex-$G$-spaces over $X$}:
Its objects are triples $(Y,p,\sigma)$ where $(Y,p)$ is a $G$-space
over $B$ and $\sigma$ is a section of $p$,
i.e., $p\circ\sigma$ is the identity.
A map $(Y,p,\sigma)\to (Z,q,\tau)$ is a section-preserving $G$-map
$f\colon Y\to Z$ over $B$, i.e., a $G$-map over and under $B$.
\end{enumerate}
When the meaning is clear, we shall write just $Y$ for $(Y,p)$ or $(Y,p,\sigma)$. If $(Y,p)$ is a space over $X$, we write
$(Y,p)_+$ for the ex-$G$-space obtained by adjoining a disjoint section.
We shall also write $Y_+$ for $(Y,p)_+$.
\end{definition}

We define homotopies in these categories using general maps $Y\times I \to X$,
so homotopies are in general not fiberwise homotopies.
These are the cylinders in natural model category structures discussed in \cite{MaySig:parametrized};
we write $hG\ParaU X$ and $hG\Para X$ for the corresponding homotopy categories,
in the model category sense.

There is also a model category of {\em $G$-spectra} parametrized by $X$, which we denote
$G\Spec{}{X}$. The details are discussed in \cite{MaySig:parametrized}.
(We shall always use a complete $G$-universe to index spectra, so omit the universe
from our notation and any further mention.)
We use $hG\Spec{}{X}$ to denote the homotopy category of spectra over $X$,
and will usually write $[E,F]^G_X$ for $hG\Spec{}{X}(E,F)$.

The following definition appeared in \cite{CMW:orientation} and is variation
of tom Dieck's original definition \cite[I\S10]{tD:transfgroups}.

\begin{definition}
The {\em fundamental groupoid} of $X$ is the category whose objects are the
orbits $p\colon G/H\to X$ over $X$ and whose morphisms $p\to q$, where $q\colon G/K\to X$,
are the pairs $(\alpha,\omega)$ where $\alpha\colon G/H\to G/K$ and
$\omega\colon p\to q\alpha$ is a homotopy class of paths $G/H\times I\to X$ rel endpoints.
Composition is induced by composition of maps of orbits and composition of path classes:
If
\begin{align*}
 (\alpha,\omega)\colon & (p\colon G/H\to X) \to (q\colon G/K\to X) \qquad\text{and}\\
 (\beta,\eta)\colon & (q\colon G/K\to X) \to (r\colon G/L\to X),
\end{align*}
then
\[
 (\beta,\eta)\circ(\alpha,\omega) = (\beta\alpha, \eta\alpha * \omega) \colon p\to r.
\]
(We write composition of paths from right to left for consistency with the usual
convention for composition of morphisms.)
\end{definition}

Notice that $\Pi_G (*) = \orb G$, the orbit category of $G$. 
The projection $X\to *$ induces a functor $\phi\colon\Pi_G X \to \orb G$
that gives $\Pi_G X$ the structure of a {\em parametrized groupoid}, 
meaning that it satisfies a certain collection of axioms 
(given in \cite{CW:duality} and \cite{CMW:orientation}). 
We shall not need these axioms in this paper.
Note that $\phi^{-1}(G/H)$, the subcategory of objects mapping to $G/H$
and morphisms mapping to the identity, is $\Pi X^H$, the nonequivariant fundamental
groupoid of $X^H$.

%We could also consider the full subcategory of $hG\ParaU X$ generated by the orbits
%$G/H\to X$. This is not the fundamental groupoid but the {\em homotopy fundamental groupoid}
%$h\Pi_G X$ also considered by tom Dieck. 
%We shall not make use of this category here.

$\Pi_G X$ plays the role equivariantly that the fundamental group(oid) does nonequivariantly.
The analogue of the group ring $\Z\pi_1 X$ is the following related category.

\begin{definition}
The {\em stable fundamental groupoid} $\stab\Pi_G X$ of $X$ is the full subcategory
of $hG\Spec{}{X}$ on the suspensions of the orbits $p\colon G/H\to X$.
\end{definition}

Thus, we can think of $\stab\Pi_G X$ as having the same objects as $\Pi_G X$, but its maps are stable $G$-maps over $X$.
As a result, it is enriched over abelian groups, i.e., it is a preadditive category.
Calculationally, it can be described as a category of fractions on $\Pi_G X$:
The group of morphisms from $x\colon G/H\to X$ to $y\colon G/K\to X$ is the free abelian
group on equivalence classes of diagrams of the form $[x \from z \laxto y]$,
where $z\colon G/L\to X$, the map $z\to x$ is a strict map over $X$, and
the map $z\laxto y$ is a {\em lax} map \cite[2.2.2]{CW:homologybook}, meaning
a pair $(\alpha,\lambda)$, where $\alpha\colon G/L\to G/K$ and $\lambda$ is
a Moore path from $y\circ\alpha$ to $z$.
Only certain subgroups $L$ appear in the compact Lie case.
(See \cite[2.6.4]{CW:homologybook}.)
For a general compact
Lie group $G$, composition is tricky to describe in these terms.

%Geometrically, if $x\colon G/H\to X$ and $y\colon G/K\to X$ are objects of $\stab\Pi_G X$,
%we can describe a morphism from $x$ to $y$ as follows.
%Let $V$ be a representation of $G$ and consider the space $S_X^{V,x}$ defined as the pushout in
%the following diagram:
%\[
% \xymatrix{
%  G/H \ar[r]^x \ar[d] & X \ar[d] \\
%  G/H\times S^{V} \ar[r] & S_X^{V,x}
% }
%\]
%$S_X^{V,x}$ comes with a map $\bar x\colon S_X^{V,x}\to X$.
%Then a morphism from $x$ to $y$ is the stable equivalence class of a pair $(f,\lambda)$,
%where $f$ is a map from $S_X^{V,x}$ to $S_X^{V,y}$, and $\lambda$ is
%a Moore path from $\bar x$ to $\bar y\circ f$.

We have $\stab\Pi_G(*) = \sorb G$, the stable orbit category of $G$, and the projection
$X\to *$ induces a map $\stab\Pi_G X\to \sorb G$ for any $X$.
The description of $\sorb G$ as a category of fractions is
\cite{LMS:eqhomotopy} Corollary V.9.4 and subsequent discussion.
Note, in particular, that we can identify $\sorb G(G/H,G/G) \iso \sorb H(H/H,H/H)$ with $A(H)$,
the Burnside ring of $H$.

As mentioned above, we need to consider local coefficients. Equivariantly, there are two aspects
to this: a more general notion of grading and a general notion of coefficient system. We start with
the grading.

\begin{definition}
Let $\ortho G{}$ be the category whose objects are the 
orthogonal $G$-vector bundles over orbits of
$G$ and whose morphisms are the equivalence classes of
$G$-vector bundle maps between them. Here, two maps are equivalent if they are $G$-bundle
homotopic over the constant homotopy on base spaces.
Let $\psi\colon \ortho G{} \to \orb G$ be the functor taking the bundle $p\colon E\to G/H$ to 
its base space $G/H$, and taking a bundle
map to the underlying map of base spaces. Let $\ortho Gn$ be the full
subcategory of $\ortho G{}$ consisting of the $n$-dimensional bundles.
 \end{definition}

A $G$-vector bundle $\xi\colon E\to X$ determines a map $\xi^*\colon \Pi_G X \to
\ortho G{}$ over $\orb G$, with $\xi^*(p\colon G/H\to X) = p^*(\xi)$.
The map $\xi^*$ is an example of an {\em orthogonal representation} of
the fundamental groupoid $\Pi_G X$, which is just a map $\Pi_G X\to \ortho G{}$ over $\orb G$. We think of $\xi^*$ (actually, its natural isomorphism class) as the {\em dimension} of $\xi$.

If $V$ is any representation of $G$, there is a representation of any $\Pi_G X$ given by sending $p$ to $\phi(p)\times V$. We shall call this representation (of $\Pi_G X$) $V$ again.
If $M$ is a smooth $G$-manifold, its {\em tangent representation} $\tau$ is the representation of $\Pi_G M$ associated with the tangent bundle $TM$ of $M$.
Note that, if the fixed sets are not all orientable, $\tau$ will encode some twisting
data as we go around loops in $\Pi_G M$.

In \cite{CMW:orientation} we construct various other categories of bundles over orbits. 
In particular, we can work with {\em virtual representations}, which (for fundamental groupoids of compact spaces) are just the usual formal differences of representations.
It is the virtual representations of $\Pi_G X$ that we will use to grade the ordinary homology
of spaces over $X$.

Nonequivariantly, local coefficients can be defined as modules over the group ring
$\Z\pi_1 X$. Equivariantly, they are functors on the stable fundamental groupoid.

\begin{definition}
A {\em $\stab\Pi_G X$-module} is an additive functor from $\stab\Pi_G X$ to the category
of abelian groups. We shall consider both contravariant and covariant modules and adopt
the notational convention that contravariant modules will be written with a
bar on top, as $\Mackey T$, while covariant modules will be written with an underline,
as $\MackeyOp S$.
\end{definition}

We are now ready to introduce the kinds of cell complexes we will be working with.
Let $p\colon Y\to X$ be a $G$-space over $X$ and let $\gamma$ be a representation
of $\Pi_G X$.

\begin{definition}
A {\em $G$-CW($\gamma$) structure} on $Y$ is a decomposition
\[
 (Y,p) = \colim_n (Y^n,p^n)
\]
in $G\ParaU X$, where
\begin{enumerate}
\item
$Y^0$ is a disjoint union of orbits $(G/H,q)$ for which $\gamma(q) \iso G/H\times \Real^k$
for some $k$,
i.e., $H$ acts trivially on the fiber of $\gamma(q)$, and
\item
each $(Y^n,p^n)$, for $n>0$, is obtained from $(Y^{n-1},p^{n-1})$ by attaching cells of the
form $(G\times_H D(V),q)$ along maps $(G\times_H S(V),q)\to (Y^{n-1},p^{n-1})$, where
$|V| = n$ and $\gamma(q|G/H\times 0) \iso G\times_H(V\pm\Real^k)$ for some $k$.
\end{enumerate}
Given a particular $G$-CW($\gamma$) structure on $Y$, we say that $Y$ is a
{\em $G$-CW($\gamma$) complex.}
We will often write $Y^{\gamma+k}$ for the skeleton $Y^{|\gamma|+k}$.
\end{definition}

Thus, a $G$-CW($\gamma$) complex is one built out of cells that are locally modeled
on $\gamma$.

For the following definition, let $\Lie(G/H)$ denote the tangent plane to $G/H$ at $eH$,
as a representation of $H$.
We think of $\Lie(G/H)$ as the dimension of the manifold $G/H$.

\begin{definition}
A {\em dual $G$-CW($\gamma$) structure} on $Y$ is a decomposition
\[
 (Y,p) = \colim_n (Y^n,p^n)
\]
in $G\ParaU X$, where
\begin{enumerate}
\item
$Y^0$ is a disjoint union of orbits $(G/H,q)$ for which $\Lie(G/H) = 0$ (i.e., $G/H$ is finite) and
$\gamma(q) \iso G/H\times \Real^k$
for some $k$, and
\item
each $(Y^n,p^n)$, for $n>0$, is obtained from $(Y^{n-1},p^{n-1})$ by attaching cells of the
form $(G\times_H D(V),q)$ along maps $(G\times_H S(V),q)\to (Y^{n-1},p^{n-1})$, where
$|V + \Lie(G/H)| = n$ and $\gamma(q|G/H\times 0) \iso G\times_H(V+\Lie(G/H)\pm\Real^k)$ for some $k$.
\end{enumerate}
Given a particular dual $G$-CW($\gamma$) structure on $Y$, we say that $Y$ is a
{\em dual $G$-CW($\gamma$) complex.}
We will often write $Y^{\gamma+k}$ for the skeleton $Y^{|\gamma|+k}$.
\end{definition}

The difference between $G$-CW and dual $G$-CW complexes is this: In an ordinary $G$-CW complex,
we think of a cell $G\times_H D(V)$ as being $V$-dimensional, ignoring $G/H$.
In a dual complex, such a cell is thought of as $(V+\Lie(G/H))$-dimensional, which is the
geometric dimension of the manifold $G\times_H D(V)$.
Of course, if $G$ is finite, there is no difference.

We define relative $G$-CW($\gamma$) and relative dual $G$-CW($\gamma$) complexes in the obvious way.
A based complex is an ex-space $(Y,p,\sigma)$ such that $(Y,\sigma(X))$ is a relative complex.
These complexes are discussed in detail in \cite{CW:homologybook}, where the expected results are shown,
including suitable approximation and Whitehead theorems.
It also follows from results there that it is unambigous to say that $Y$
has the homotopy type of a finite complex, because a finite $G$-CW($\gamma$) complex
has the homotopy type of a finite dual $G$-CW($\delta$) complex and vice versa,
for any $\gamma$ and $\delta$.

A crucial example we have in mind is that of a compact smooth $G$-manifold $M$.
In \cite{Ill:triangulation}, Illman showed that every manifold has a $G$-triangulation,
which can be viewed as a $G$-CW(0) structure (over a point or over $M$ itself).
We can also consider the cell structure dual to the triangulation;
considering $M$ as a space over itself, this structure is a dual $G$-CW($\tau$) structure,
where $\tau$ is the tangent representation.

Given a $G$-CW complex, we define its chain complex.
In the following, if $q\colon G/H\to X$ is an orbit over $X$
and $\xi$ is a $G$-vector bundle over $G/H$,
$S_X^{\xi,q}$ denotes the
ex-space over $X$ given as the pushout in the following diagram:
\[
 \xymatrix{
  G/H \ar[r]^q \ar[d] & X \ar[d] \\
  S^{\xi} \ar[r] & S_X^{\xi,q}
 }
\]
Here, $S^{\xi}$ denotes the space obtained by taking the fiberwise one-point compactification
of $\xi$ and the map $G/H\to S^{\xi}$ is the inclusion of
the compactification points.
If $V$ is a representation of $H$, we write $\xi+V$ as shorthand for the bundle
$G\times_H (\xi_0+V)$, where $\xi = G\times_H \xi_0$.
In particular, we shall write $S_X^{V,q}$ as shorthand for 
the ex-space obtained as in the diagram above with $\xi = G\times_H V$.

\begin{definition}
Let $Y$ be a $G$-CW($\gamma$) complex over $X$.
The {\em cellular chain complex} of $Y$ is the chain complex $\Mackey C^G_{\gamma+*}(Y)$
of contravariant $\stab\Pi_G X$-modules defined by
\[
 \Mackey C^G_{\gamma+k}(Y)(q) 
  = [\susp_X^\infty S^{\gamma(q)+k,q}_X, \susp_X^\infty Y^{\gamma+k}/_X Y^{\gamma+k-1}]^G_X
\]
for $|\gamma|+k \geq 0$,
where ``$/_X$'' denotes the fiberwise quotient over $X$.
\end{definition}

It takes a bit of work, done in \cite{CW:homologybook},
to show that this defines a functor on $\stab\Pi_G X$.
In fact, each $\Mackey C^G_{\gamma+k}(Y)$ is a {\em free} module, in the sense that
it is a direct sum of modules of the form $\stab\Pi_G X(-,p)$, where $p$ ranges
over the centers of the $(|\gamma|+k)$-cells of $Y$.

For dual complexes we have a similar definition.
We use the crucial duality, shown in \cite{CW:homologybook} or \cite{MaySig:parametrized},
that, if $q\colon G/H\to X$ and $r\colon G/K\to X$, then
\[
 [\susp_X^\infty S_X^{V-\Lie(G/H),q}, \susp_X^\infty S_X^{V-\Lie(G/K),r}]^G_X
  \iso [\susp_X^\infty (G/K,r)_+, \susp_X^\infty (G/H,q)_+]^X_G
\]
for $V$ large enough to contain copies of $\Lie(G/H)$ and $\Lie(G/K)$.
In this sense, the stable dual over $X$ of $q\colon G/H\to X$
is $S_X^{-\Lie(G/H),q}$.

\begin{definition}
Let $Y$ be a dual $G$-CW($\gamma$) complex over $X$.
The {\em dual cellular chain complex} of $Y$ is the chain complex $\MackeyOp C^G_{\gamma+*}(Y)$
of covariant $\stab\Pi_G X$-modules defined by
\[
 \MackeyOp C^G_{\gamma+k}(Y)(q) 
  = [\susp_X^\infty S^{\gamma(q)-\Lie(G/H)+k,q}_X, \susp_X^\infty Y^{\gamma+k}/_X Y^{\gamma+k-1}]^G_X
\]
for $|\gamma|+k \geq 0$.
\end{definition}

That this gives covariant $\stab\Pi_G X$-modules follows from the duality mentioned above.
Again, they are free modules generated by the centers of the dual cells.

It is now easy to define homology and cohomology groups. In the following definition, if
$\Mackey C$ is a contravariant $\stab\Pi_G X$-module and $\MackeyOp S$ is a covariant module,
then the tensor product is given by a coend:
\[
 \Mackey C \tensor_{\stab\Pi_G X} \MackeyOp S
  = \int^{q\in\stab\Pi_G X} \Mackey C(q)\tensor \MackeyOp S(q).
\]

\begin{definition}
Let $Y$ be a $G$-CW($\gamma$) complex over $X$, let $\MackeyOp S$ be a covariant
$\stab\Pi_G X$-module, and let $\Mackey T$ be a contravariant $\stab\Pi_G X$-module.
Then we let
\[
 H^G_{\gamma+k}(Y;\MackeyOp S) = 
  H_{\gamma+k}(\Mackey C^G_{\gamma+*}(Y)\tensor_{\stab\Pi_GX}\MackeyOp S)
\]
and
\[
 H_G^{\gamma+k}(Y;\Mackey T) =
  H^{\gamma+k}\Hom_{\stab\Pi_GX}(\Mackey C^G_{\gamma+*}(Y), \Mackey T).
\]
We call these the {\em ordinary homology and cohomology of $Y$}.
Similarly, if $Y$ is a dual $G$-CW($\gamma$) complex, we let
\[
 \H^G_{\gamma+k}(Y;\Mackey T) =
  H_{\gamma+k}(\MackeyOp C^G_{\gamma+*}(Y) \tensor_{\stab\Pi_GX} \Mackey T)
\]
and
\[
 \H_G^{\gamma+k}(Y;\MackeyOp S) =
  H^{\gamma+k}\Hom_{\stab\Pi_GX}(\MackeyOp C^G_{\gamma+*}(Y), \MackeyOp S).
\]
We call these the {\em dual ordinary homology and cohomology of $Y$}.
\end{definition}

We define relative and reduced homology and cohomology in the obvious ways,
by taking quotient chain complexes.

In \cite{CW:homologybook} we show that these groups satisfy the axioms that
qualify them as equivariant homology and cohomology theories on spaces over $X$,
including suspension isomorphisms that allow us to consider them, for a fixed $\gamma$,
as $RO(G)$-graded. But, we think of them as ``$RO(\Pi_G X)$''-graded,
where $RO(\Pi_G X)$ is the group of virtual representations of $\Pi_GX$.
For $X$ compact, every such virtual representation can be written as $\gamma-V$
for a representation $V$, so the suspension isomorphisms suffice.
For more general $X$ it takes more work, done in \cite{CW:homologybook}.

We call these theories ordinary because they also satisfy dimension axioms.
Those axioms take the following forms, in which all isomorphisms are natural
in $(G/H,q)\in \stab\Pi_GX$ and $k$ is an integer:
\begin{align*}
 \tilde H^G_{\gamma+k}(S_X^{\gamma(q),q};\MackeyOp S)
  &\iso \begin{cases}
  		\MackeyOp S(q) & \text{if $k = 0$} \\
		0 & \text{if $k\neq 0$}
	   \end{cases} \\
 \tilde H_G^{\gamma+k}(S_X^{\gamma(q),q};\Mackey T)
  &\iso \begin{cases}
  		\Mackey T(q) & \text{if $k = 0$} \\
		0 & \text{if $k\neq 0$}
	   \end{cases} \\
 \tilde \H^G_{\gamma+k}(S_X^{\gamma(q)-\Lie(G/H),q};\Mackey T)
  &\iso \begin{cases}
  		\Mackey T(q) & \text{if $k = 0$} \\
		0 & \text{if $k\neq 0$}
	   \end{cases} \\
 \tilde \H_G^{\gamma+k}(S_X^{\gamma(q)-\Lie(G/H),q};\MackeyOp S)
  &\iso \begin{cases}
  		\MackeyOp S(q) & \text{if $k = 0$} \\
		0 & \text{if $k\neq 0$}
	   \end{cases} \\
\end{align*}
In the dual cases, we suspend the left-hand sides sufficiently so that
subtracting $\Lie(G/H)$ makes sense, and we use duality to reverse the variance
and interpret those left-hand sides as functors on $\stab\Pi_GX$.
Using the fact that $S_X^{\gamma(q)-\Lie(G/H),q}$ is the stable dual
of $S_X^{\gamma(q),q}$, we can see that the dual of ordinary homology is {\em dual}
ordinary cohomology, {\em not} ordinary cohomology;
similarly, the dual of ordinary cohomology is dual homology.
(See \cite[\S 2.7]{CW:homologybook} for a discussion of the homological duality
over $X$ being used here.)
Of course, this is a distinction that vanishes if $G$ is finite.

Among the many nice properties these theories have, we point out
several useful spectral sequences.
The first are the Atiyah-Hirzebruch spectral sequences, for which we need
the following definition.

\begin{definition}\label{def:coeffsystem}
Let $\tilde h^G_*$ be an $RO(G)$-graded homology theory on spaces over $X$
and let $\gamma$ be a representation of $\Pi_G X$.
For $V$ a representation of $G$,
we define the {\em $(V-\gamma)$-coefficient system} of $\tilde h^G_*$ to be
the covariant $\stab\Pi_G X$-module $\MackeyOp h^G_{V-\gamma}$ defined
on $q\colon G/H\to X$ by
\[
 \MackeyOp h^G_{V-\gamma}(q) = \tilde h^G_V(S_X^{\gamma(q),q}).
\]
We define the {\em dual $(V-\gamma)$-coefficient system}
to be the contravariant $\stab\Pi_G X$-module $\MackeyOp h^{G,\Lie}_{V-\gamma}$ defined by
\[
 \Mackey h^{G,\Lie}_{V-\gamma}(q)
  = \tilde h^G_V(S_X^{\gamma(q)-\Lie(G/H),q}).
\]
(We use duality to view this as a contravariant functor in $q$.)
Similarly, if $\tilde h_G^*$ is an $RO(G)$-graded cohomology theory on spaces over $X$,
we define coefficient systems $\Mackey h_G^{V-\gamma}$ and $\MackeyOp h_{G,\Lie}^{V-\gamma}$ by
\[
 \Mackey h_G^{V-\gamma}(q) = \tilde h_G^V(S_X^{\gamma(q),q})
\]
and
\[
 \MackeyOp h_{G,\Lie}^{V-\gamma}(q)
  = \tilde h_G^V(S_X^{\gamma-\Lie(G/H),q}).
\]
\end{definition}

If $\tilde h^G_*$ and $\tilde h_G^*$ are represented by the same spectrum over $X$,
then the duality theory of \cite[\S 2.7]{CW:homologybook}
shows that
\[
 \MackeyOp h^G_{V-\gamma} \iso \MackeyOp h_{G,\Lie}^{V-\gamma}
 \quad\text{and}\quad
 \Mackey h^{G,\Lie}_{V-\gamma} \iso \Mackey h_G^{V-\gamma}.
\]

The following is \cite[Theorem 3.3.16]{CW:homologybook}.
It comes from analyzing the exact couples obtained by applying $\tilde h$
to the skeletal filtrations of ordinary or dual $G$-CW($\gamma$) complexes.
We state the result for reduced theories, which implies analogous results
for the unreduced case.

\begin{theorem}[Atiyah-Hirzebruch Spectral Sequences]
Let $\tilde h^G_*$ be an $RO(G)$-graded homology theory on spaces over $X$
and let $Y$ be an ex-space over $X$.
Then we have strongly convergent spectral sequences
\[
 E^2_{p,q} = \tilde H^G_{\gamma+p}(Y;\MackeyOp h^G_{q-\gamma})
  \convto \tilde h^G_{p+q}(Y)
\]
and
\[
 E^2_{p,q} = \tilde\H^G_{\gamma+p}(Y;\Mackey h^{G,\Lie}_{q-\gamma})
  \convto \tilde h^G_{p+q}(Y).
\]
Similarly, if $\tilde h_G^*$ is an $RO(G)$-graded cohomology theory on spaces over $X$,
we have conditionally convergent spectral sequences
\[
 E_2^{p,q} = \tilde H_G^{\gamma+p}(Y;\Mackey h_G^{q-\gamma})
  \convto \tilde h_G^{p+q}(Y)
\]
and 
\[
 E_2^{p,q} = \tilde\H_G^{\gamma+p}(Y;\MackeyOp h_{G,\Lie}^{q-\gamma})
  \convto \tilde h_G^{p+q}(Y).
\]
\qed
\end{theorem}

Also useful are the universal coefficients spectral sequences in which
the universal coefficients are the free $\stab\Pi_G X$-modules.
We introduce the following notation, which we will use periodically throughout this paper.
We switch to unreduced theories, which is the case we will use most,
with analogous notations and results for the relative and reduced cases.

\begin{definition}\label{def:universalcoeffs}
For $Y$ a $G$-space over $X$, write
$\Mackey H^G_*(Y)$ for the contravariant $\stab\Pi_GX$-module whose value at $q$ is
\[
 \Mackey H^G_*(Y)(q) = H^G_*(Y;\stab\Pi_GX(q,-)).
\]
Similarly, we have modules $\MackeyOp\H^G_*(Y)$, $\MackeyOp H_G^*(Y)$, and
$\Mackey\H_G^*(Y)$ defined by
\begin{align*}
 \MackeyOp\H^G_*(Y)(q) &= \H^G_*(Y;\stab\Pi_GX(-,q)),\\
 \MackeyOp H_G^*(Y)(q) &=  H_G^*(Y;\stab\Pi_GX(-,q)),\text{ and}\\
 \Mackey \H_G^*(Y)(q) &= \H_G^*(Y;\stab\Pi_GX(q,-)).
\end{align*}
\end{definition}

Note that $\Mackey H^G_*(Y)$ is just the homology of $\Mackey C^G_*(Y)$
considered as a chain complex of $\stab\Pi_GX$-modules,
and similarly for $\MackeyOp\H^G_*(Y)(q)$.

The following is \cite[Theorem 3.3.18]{CW:homologybook}
(with some additional cases).
Here, $\Tor_*^{\stab\Pi_GX}$ and $\Ext^*_{\stab\Pi_GX}$ are the derived
functors of $\tensor_{\stab\Pi_GX}$ and $\Hom_{\stab\Pi_GX}$, respectively.

\begin{theorem}[Universal Coefficients Spectral Sequences]
Let $Y$ be a $G$-space over $X$, let $\MackeyOp S$ be a covariant
$\stab\Pi_GX$-module, and let $\Mackey T$ be a contravariant $\stab\Pi_GX$-module.
Then we have the following natural spectral sequences:
\begin{align*}
 E^2_{p,q} = \Tor_p^{\stab\Pi_GX}(\Mackey H^G_{\gamma+q}(Y),\MackeyOp S)
   &\convto  H^G_{\gamma+p+q}(Y;\MackeyOp S) \\
 E^2_{p,q} = \Tor_p^{\stab\Pi_GX}(\MackeyOp\H^G_{\gamma+q}(Y),\Mackey T)
   &\convto \H^G_{\gamma+p+q}(Y;\Mackey T) \\
 E_2^{p,q} = \Ext^p_{\stab\Pi_GX}(\Mackey H^G_{\gamma+q}(Y),\Mackey T)
   &\convto  H_G^{\gamma+p+q}(Y;\Mackey T) \\
 E_2^{p,q} = \Ext^p_{\stab\Pi_GX}(\MackeyOp\H^G_{\gamma+q}(Y),\MackeyOp S)
   &\convto \H_G^{\gamma+p+q}(Y;\MackeyOp S).
\end{align*}
If $Y$ is a finite ordinary or dual $G$-CW($\gamma$) complex, as appropriate, we also
have the following natural spectral sequences:
\begin{align*}
 E_2^{p,q} = \Ext^p_{\stab\Pi_GX}(\MackeyOp H_G^{\gamma-q}(Y),\MackeyOp S)
   &\convto  H^G_{\gamma-p-q}(Y;\MackeyOp S) \\
 E_2^{p,q} = \Ext^p_{\stab\Pi_GX}(\Mackey\H_G^{\gamma-q}(Y),\Mackey T)
   &\convto \H^G_{\gamma-p-q}(Y;\Mackey T) \\
 E^2_{p,q} = \Tor_p^{\stab\Pi_GX}(\MackeyOp H_G^{\gamma-q}(Y),\Mackey T)
   &\convto  H_G^{\gamma-p-q}(Y;\Mackey T) \\
 E^2_{p,q} = \Tor_p^{\stab\Pi_GX}(\Mackey\H_G^{\gamma-q}(Y),\MackeyOp S)
   &\convto \H_G^{\gamma-p-q}(Y;\MackeyOp S).
\end{align*}
\qed
\end{theorem}

Combining univeral coefficients and the Atiyah-Hirzebruch spectral sequence
gives us a very useful corollary.

\begin{corollary}\label{cor:universalisomorphism}
Let $f\colon Y\to Z$ be a $G$-map of finite based $G$-CW complexes over $X$. If any one of
\begin{align*}
 f_*\colon \Mackey H^G_{\gamma+*}(Y) &\to \Mackey H^G_{\gamma+*}(Z) \\
 f_*\colon \MackeyOp\H^G_{\gamma+*}(Y) &\to \MackeyOp\H^G_{\gamma+*}(Z) \\
 f^*\colon \MackeyOp H_G^{\gamma+*}(Z) &\to \MackeyOp H_G^{\gamma+*}(Y) \text{ or}\\
 f^*\colon \Mackey\H_G^{\gamma+*}(Z) &\to \Mackey\H_G^{\gamma+*}(Y)
\end{align*}
is an isomorphism (with $*\in\Z$), then so are the other three and so are all induced maps
in ordinary and dual homology and cohomology in these gradings, with any coefficients.
\qed
\end{corollary}

Finally, we record an algebraic lemma, a slight generalization of
\cite[Lemma 2.3]{Wal:surgery}.
Recall that we say that a $\stab\Pi_G X$-module $\Mackey C$ is free if it is a direct sum
of represented modules, say
\[
 \Mackey C \iso \Dirsum_i \stab\Pi_G X(-,p_i).
\]
We call the indexed collection $\{p_i\}$ a {\em basis} for $\Mackey C$
and say that the basis is {\em $G$-free} if each $p_i\colon G/e\to X$.

\begin{lemma}\label{lem:stablyfree}
Suppose that $\Mackey C_*$ is a finitely-generated
nonnegatively-graded chain complex of free $\stab\Pi_G X$-modules,
and assume that there is an integer $\mu\geq 0$ such that
\begin{enumerate}
\item $H_k(\Mackey C_*) = 0$ for all $k\neq\mu$ and
\item $H^{\mu+1}(\Mackey C_*;\Mackey T) = 0$ for all $\stab\Pi_G X$-modules $\Mackey T$.
\end{enumerate}
Then $H_\mu(\Mackey C_*)$ is a finitely generated stably free $\stab\Pi_G X$-module.
Moreover, if the basis elements for $\Mackey C_*$ are all $G$-free, then
$H_\mu(\Mackey C_*)$ has a $G$-free stable basis.
\end{lemma}

\begin{proof}
All but the last sentence of the proof is 
essentially the same as that of \cite[Lemma 2.3]{Wal:surgery},
but we repeat it here for that last statement.

Let $\Mackey Z_k$ and $\Mackey B_k$ be the submodules of $\Mackey C_k$ 
consisting of the cycles and boundaries, respectively.
For $k \leq \mu$ we have short exact sequences
\[
 0 \to \Mackey Z_k \to \Mackey C_k \to \Mackey Z_{k-1} \to 0
\]
and we have $\Mackey Z_0 = \Mackey C_0$. By induction, 
the above short exact sequence splits and $\Mackey Z_k$ is projective
for every $k\leq\mu$.
It follows that $\Mackey C_*$ is chain homotopy equivalent to the complex
\[
 \cdots \to \Mackey C_{\mu+2} \to \Mackey C_{\mu+1} \to \Mackey Z_\mu \to 0
\]
and that the following complex of projective modules has 0 homology, hence is contractible:
\begin{equation}\label{seq:bottomhalf}
 0 \to \Mackey Z_{\mu-1} \to \Mackey C_{\mu-1} \to \Mackey C_{\mu-2} \to \cdots \to \Mackey C_0 \to 0.
\end{equation}
By assumption, the cocycle $\Mackey C_{\mu+1}\to \Mackey B_\mu$ is a coboundary,
which gives a splitting $\Mackey Z_\mu \to \Mackey B_\mu$ to the inclusion,
hence $\Mackey B_\mu$ is also projective.
This gives us that $H_\mu(\Mackey C_*) \dirsum \Mackey B_\mu \iso \Mackey Z_\mu$,
so $H_\mu(\Mackey C_*)$ is projective and finitely generated.
We also get that the following is a chain complex of projective modules with 0 homology,
hence is contractible:
\begin{equation}\label{seq:tophalf}
 \cdots \to \Mackey C_{\mu+2} \to \Mackey C_{\mu+1} \to \Mackey B_\mu \to 0.
\end{equation}
The direct sum of complexes~(\ref{seq:bottomhalf}) and~(\ref{seq:tophalf}) is a contractible
complex of projective modules, hence
\[
 \Dirsum_{i}\Mackey C_{\mu+2i+1} 
   \iso \Mackey B_\mu\dirsum \Mackey Z_{\mu-1}\Dirsum_{i\neq 0}\Mackey C_{\mu+2i}.
\]
Adding $H_\mu(\Mackey C^*)$ to both sides, and using
\[
 H_\mu(\Mackey C^*)\dirsum \Mackey B_\mu\dirsum \Mackey Z_{\mu-1}
  \iso \Mackey Z_\mu\dirsum \Mackey Z_{\mu-1}
  \iso \Mackey C_\mu,
\]
we get
\[
 H_\mu(\Mackey C_*) \dirsum \Dirsum_{i}\Mackey C_{\mu+2i+1}
  \iso \Dirsum_{i}\Mackey C_{\mu+2i}.
\]
This shows that $H_\mu(\Mackey C_*)$ is stably free.
If we assume that $\Mackey C_*$ has a $G$-free basis, then the 
isomorphism above exhibits a stable basis of
$H_\mu(\Mackey C_*)$ that is also $G$-free.
\end{proof}

%%%%%%%%%%%%%%%%%%%%%%%%%%%%%%%%%
\section{Equivariant Poincar\'e Duality Spaces}

The homology and cohomology theories described in the previous section have
cup and cap products, though care has to be taken with the coefficient systems,
as discussed in detail in \cite{CW:homologybook}. 
The special case we need is the following.
Let $\Mackey A_{G/G}$ denote the coefficient system on any $X$ defined by
\[
 \Mackey A_{G/G}(p\colon G/H\to X) = \sorb G(G/H,G/G) \iso A(H).
\]
Then, if $X$ is a $G$-space and $A$ and $B$ are subspaces such that
$(A\union B; A, B)$ is an excisive triad (or, $A$ and $B$ are subcomplexes of $X$),
we have cap products
\[
 -\cap -\colon H_G^{\beta}(X,A;\Mackey T)\tensor \H^G_{\gamma}(X,A\union B;\Mackey A_{G/G})
  \to \H^G_{\beta-\gamma}(X,B;\Mackey T)
\]
and
\[
 -\cap -\colon \H_G^{\beta}(X,A;\MackeyOp S)\tensor \H^G_{\gamma}(X,A\union B;\Mackey A_{G/G})
  \to H^G_{\beta-\gamma}(X,B;\MackeyOp S)
\]
for any coefficient systems $\Mackey T$ and $\MackeyOp S$.

The following is \cite[Theorem 3.8.8]{CW:homologybook}.

\begin{theorem}[Lefschetz Duality]
Let $M$ be a compact smooth $G$-manifold and let $\tau$ be its tangent representation.
Then there exists a {\em fundamental class}
$[M,\bndry M]\in \H^G_\tau(M,\bndry M;\Mackey A_{G/G})$
such that the following maps are all isomorphisms, for every $\gamma$ and
every coefficient system:
\begin{align*}
 -\cap [M,\bndry M]\colon{}& H_G^\gamma(M;\Mackey T) \to \H^G_{\tau-\gamma}(M,\bndry M;\Mackey T) \\
 -\cap [M,\bndry M]\colon{}& H_G^\gamma(M,\bndry M;\Mackey T)\to \H^G_{\tau-\gamma}(M;\Mackey T) \\
 -\cap [M,\bndry M]\colon{}& \H_G^\gamma(M;\MackeyOp S) \to H^G_{\tau-\gamma}(M,\bndry M;\MackeyOp S) \\
 -\cap [M,\bndry M]\colon{}& \H_G^\gamma(M,\bndry M;\MackeyOp S)\to H^G_{\tau-\gamma}(M;\MackeyOp S).
\end{align*}
Moreover, $\bndry[M,\bndry M]\in \H^G_{\tau-1}(\bndry M;\Mackey A_{G/G})$ is
a fundamental class for $\bndry M$, in that the following maps are isomorphisms, for every
$\gamma$ and every coefficient system:
\begin{align*}
 -\cap \bndry[M,\bndry M]\colon{}& H_G^\gamma(\bndry M;\Mackey T) 
     \to \H^G_{\tau-\gamma-1}(\bndry M;\Mackey T) \text{ and}\\
 -\cap \bndry[M,\bndry M]\colon{}& \H_G^\gamma(\bndry M;\MackeyOp S)
     \to H^G_{\tau-\gamma-1}(\bndry M;\MackeyOp S).
\end{align*}
\qed
\end{theorem}

Fundamental classes are defined in \cite{CW:homologybook}
by the property that, for each point $m\in M-\bndry M$, the restriction to
\[
 \H^G_\tau(M,M-Gm;\Mackey A_{G/G}) \iso A(G_m)
\]
is a generator. (It is this requirement that forces the fundamental class
to live in dual homology.)
It is also shown there that a fundamental class is characterized by its restrictions to subgroups and to fixed sets,
as in the following.
(These restriction maps are discussed in detail in that source.)

\begin{proposition}
Let $M$ be a compact smooth $G$-manifold and let $\tau$ be the tangent representation
of $\Pi_G M$. Then each of the following is equivalent to
$[M,\bndry M] \in \H^G_\tau(M,\bndry M;\Mackey A_{G/G})$ being a fundamental class:
\begin{enumerate}
\item
For every $K\leq G$, $[M,\bndry M]|K \in \H^K_\tau(M,\bndry M;\Mackey A_{K/K})$
is a fundamental class of $M$ as a $K$-manifold.

\item
For every $K\leq G$, $[M,\bndry M]^K \in \H^{WK}_{\tau^K}(M^K,\bndry M^K;\Mackey A_{WK/WK})$
is a fundamental class of $M^K$ as a $WK$-manifold.

\item
For every $K\leq G$, $[M,\bndry M]^K|e \in H_{\tau^K}(M^K,\bndry M^K;\Z)$
is a fundamental class of $M^K$ as a nonequivariant manifold (where $H_{\tau^K}$ is
twisted homology in dimension $|\tau^K|$ if $M^K$ is not orientable).
\qed
\end{enumerate}
\end{proposition}

With these results in mind, we make the following definition.
We use the notation introduced in Definition~\ref{def:universalcoeffs}.

\begin{definition}
Let $(X,\bndry X)$ be a pair of $G$-spaces and let $\tau$ be a representation of $\Pi_G X$.
We say that $(X,\bndry X)$ is a {\em $G$-Poincar\'e duality pair of dimension $\tau$}
if it is equivalent to a pair of finite $G$-CW($\tau$) complexes
and there exists a class $[X,\bndry X]\in \H^G_\tau(X,\bndry X;\Mackey A_{G/G})$ such that
\begin{align*}
 -\cap [X,\bndry X]\colon{}& \MackeyOp H_G^\gamma(X) \to \MackeyOp\H^G_{\tau-\gamma}(X,\bndry X)
   \text{ and}\\
 -\cap \bndry[X,\bndry X]\colon{}& \MackeyOp H_G^\gamma(\bndry X)
  \to \MackeyOp\H^G_{\tau-\gamma-1}(\bndry X)
\end{align*}
are isomorphisms for every $\gamma$.
\end{definition}

We call $[X,\bndry X]$ a fundamental class for $(X,\bndry X)$.
When $\bndry X = \emptyset$, we call $X$ a $G$-Poincar\'e duality space.
The assumptions in the definition imply the other isomorphisms we would like to have:

\begin{proposition}
If $(X,\bndry X)$ is a $G$-Poincar\'e duality pair of dimension $\tau$ with
fundamental class $[X,\bndry X]$, then the following are isomorphisms
for all coefficient systems:
\begin{align*}
 -\cap [X,\bndry X]\colon{}& H_G^\gamma(X;\Mackey T) \to \H^G_{\tau-\gamma}(X,\bndry X;\Mackey T)\\
 -\cap [X,\bndry X]\colon{}& H_G^\gamma(X,\bndry X;\Mackey T)\to \H^G_{\tau-\gamma}(X;\Mackey T) \\
 -\cap [X,\bndry X]\colon{}& \H_G^\gamma(X;\MackeyOp S) \to H^G_{\tau-\gamma}(X,\bndry X;\MackeyOp S) \\
 -\cap [X,\bndry X]\colon{}& \H_G^\gamma(X,\bndry X;\MackeyOp S)\to H^G_{\tau-\gamma}(X;\MackeyOp S)\\
 -\cap \bndry[X,\bndry X]\colon{}& H_G^\gamma(\bndry X;\Mackey T)
  \to \H^G_{\tau-\gamma-1}(\bndry X;\Mackey T)\\
 -\cap \bndry[X,\bndry X]\colon{}& \H_G^\gamma(\bndry X;\MackeyOp S)
     \to H^G_{\tau-\gamma-1}(\bndry X;\MackeyOp S).
\end{align*}
Moreover, we have
\begin{enumerate}
\item
for every $K\leq G$, $(X,\bndry X)$ is a $K$-Poincar\'e duality space of dimension $\tau$
with fundamental class $[X,\bndry X]|K$;

\item
for every $K\leq G$, $(X^K,\bndry X^K)$ is a $WK$-Poincar\'e duality space of dimension $\tau^K$
with fundamental class $[X,\bndry X]^K$;

\item
for every $K\leq G$, $(X^K,\bndry X^K)$ is a nonequivariant Poincar\'e duality space
of (possibly twisted) dimension $\tau^K$ with fundamental class
$[X,\bndry X]^K|e$.
\end{enumerate}
\end{proposition}

\begin{proof}
The isomorphism $H_G^\gamma(X;\Mackey T) \iso \H^G_{\tau-\gamma}(X,\bndry X;\Mackey T)$
for every coefficient system $\Mackey T$ follows from a comparison
of universal coefficient spectral sequences.
Looking at the map $-\cap [X,\bndry X]$ on the chain level, we see that it induces a map
of spectral sequences
\[
 \xymatrix{
   \Tor^{\Pi_G X}_p(\MackeyOp H_G^{\gamma-q}(X),\Mackey T) \ar@{=>}[r] \ar[d]
     & H_G^{\gamma-p-q}(X;\Mackey T) \ar[d]^{-\cap[X,\bndry X]} \\
   \Tor^{\Pi_G X}_p(\MackeyOp \H^G_{\tau-\gamma+q}(X,\bndry X),\Mackey T) \ar@{=>}[r]
     & \H^G_{\tau-\gamma+p+q}(X,\bndry X;\Mackey T)
 }
\]
Examining the construction, we see that the map on the left is, as expected, induced by the map
\[
 -\cap[X,\bndry X]\colon \MackeyOp H_G^{\gamma-q}(X) \to \MackeyOp\H^G_{\tau-\gamma+q}(X,\bndry X)
\]
which is an isomorphism by assumption. Hence, the map on $E^2$ terms is an isomorphism,
and so is the map on the right.

A similar argument establishes the isomorphism 
$H_G^\gamma(\bndry X;\Mackey T) \iso \H^G_{\tau-\gamma-1}(\bndry X;\Mackey T)$
for all coefficient systems, and
the isomorphism $H_G^\gamma(X,\bndry X;\Mackey T)\iso \H^G_{\tau-\gamma}(X;\Mackey T)$
follows from the long exact sequence of the pair $(X,\bndry X)$.

%The isomorphism $\H_G^\gamma(X;\MackeyOp S) \iso H^G_{\tau-\gamma}(X,\bndry X;\MackeyOp S)$
%follows from another comparison of universal coefficients spectral sequences, similar
%to the one above:
%\[
% \xymatrix{
%   \Ext_{\Pi_G X}^p(\MackeyOp \H^G_{\gamma+q}(X),\MackeyOp S) \ar@{=>}[r] \ar[d]
%     & \H_G^{\gamma+p+q}(X;\MackeyOp S) \ar[d]^{-\cap[X,\bndry X]} \\
%   \Ext_{\Pi_G X}^p(\MackeyOp H_G^{\tau-\gamma-q}(X,\bndry X),\MackeyOp S) \ar@{=>}[r]
%     & H^G_{\tau-\gamma-p-q}(X,\bndry X;\MackeyOp S)
% }
%\]
%
Similar universal coefficient
and long exact sequence arguments show that
the remaining cap products induce isomorphisms.

Now let $K\leq G$. To show that $(X,\bndry X)$ is a $K$-Poincar\'e duality pair,
we use a comparison of Atiyah-Hirzebruch spectral sequences.
Consider the cohomology theory $h_G^*$ on $G$-spaces over $X$ defined by
\[
 h_G^*(Y) = H_K^*(Y;\Mackey T)
\]
where $\Mackey T$ is a $\stab\Pi_K X$-module. Similarly, let $k^G_*$ be the homology theory defined by
\[
 k^G_*(Y) = \H^K_*(Y;\Mackey T).
\]
Because homological duality is preserved when passing from $G$ to $K$, we see that we have
an isomorphism $\Mackey h_G^{\gamma+q} \iso \Mackey k^{G,\Lie}_{-\gamma-q}$,
as in the comment after Definition~\ref{def:coeffsystem}.
We can then construct the following comparison of spectral sequences:
\[
 \xymatrix{
  H_G^p(X;\Mackey h_G^{\gamma+q}) \ar@{=>}[r] \ar[d]_\iso
   & H_K^{\gamma+p+q}(X;\Mackey T) \ar[d]^{-\cap[X,\bndry X]} \\
  \H^G_{\tau-p}(X,\bndry X;\Mackey k^{G,\Lie}_{-\gamma-q}) \ar@{=>}[r]
   & \H^K_{\tau-\gamma-p-q}(X,\bndry X;\Mackey T)
 }
\]
The map of $E_2$ terms is induced by the duality isomorphism for $(X,\bndry X)$
as a $G$-Poincar\'e duality pair and the isomorphism of coefficient systems we noted.
Therefore, the map on the right is an isomorphism. Duality for $(\bndry X,\emptyset)$
is then a special case.

That $(X^K,\bndry X^K)$ is a $WK$-Poincar\'e duality pair follows by a similar
argument, using the fact that $Y \mapsto H^*_{WK}(Y^K;\MackeyOp S)$ is a homology
theory on $G$-spaces.
That $(X^K,\bndry X^K)$ is also a nonequivariant Poincar\'e duality pair follows
by combining with the previous case to forget the $WK$-action.
\end{proof}

%%%%%%%%%%%%%%%%%%%%%%%%%%%%%%%%%%%
\section{The Spivak Normal Fibration}

If a Poincar\'e duality space or pair is homotopy equivalent to a manifold, then it will have
a vector bundle over it mapping to the normal bundle of the manifold.
For any Poincar\'e duality pair, we can find a spherical fibration,
the Spivak normal fibration, that is a candidate
to be the sphere bundle of this normal bundle, as we now explain.
We concentrate on the case of a space, but similar arguments apply to pairs.

\begin{definition}\label{def:spivaknormal}
Let $X$ be a finite $G$-CW complex.
Embed $X$ equivariantly in a representation $V$ with regular neigborhood $U$. Replace the projections $p\colon U \to X$ and $q\colon \bndry U \to X$ by a pair of fibrations
$\Gamma p\colon E = \Gamma U \to X$ and $\Gamma q\colon E_0 = \Gamma \bndry U \to X$. The fiberwise quotient, $r \colon E\ParaQuot X E_0 \to X$, is then the {\em Spivak normal fibration}.
\end{definition}

We want to show that, if $X$ is a Poincar\'e duality space
and $V$ is sufficiently large, then the Spivak normal fibration
is spherical.
We will need the following results about compact Lie groups.

\begin{lemma}
Let $H$ and $K$ be subgroups of $G$. Then, as an $N_GH$-space, 
$(G/K)^H$ is a disjoint union of orbits of the form $N_GH/N_{K'}H$ where $K'$ is a
$G$-conjugate of $K$ containing $H$.
\end{lemma}

\begin{proof}
First, if $gK$ is a coset fixed by $H$, then its $N_GH$-isotropy group is
$K'\intersect N_GH = N_{K'}H$, where $K' = gKg^{-1}$ must contain $H$. 
Hence, the orbit of $gK$ has the form $N_GH/N_{K'}H$.

Suppose that $g_1K$ and $g_2K$ are connected by a path in $(G/K)^H$.
By \cite[1.1]{May:thomiso}, $g_2 = c g_1$ for some $c$ in the centralizer of $H$.
Therefore, $g_2K$ is in the $N_GH$-orbit of $g_1K$.
Thus, each component of $(G/K)^H$ is a single orbit and
$(G/K)^H$ is a disjoint union of these orbits.
\end{proof}

\begin{lemma}\label{lem:admissable}
If $V$ is a representation of $G$ and $H$ is a subgroup of $G$ with $|WH|<\infty$, then
there exists an isotropy subgroup $K$ of $V$ containing $H$ such that $V^H = V^K$ and $|WK|<\infty$.
\end{lemma}

\begin{proof}
Consider the set of all $G$-isotropy subgroups of points in $V^H$ 
(which necessarily then all contain $H$) and let $K$ be a minimal subgroup in this collection.
We claim that $|WK|<\infty$ because $H\leq K$: 
Recall that $\Lie(G/H)$ denotes the tangent space at $eH$
of $G/H$, considered as a representation of $H$.
To say that $|WH|<\infty$ is equivalent to saying that $\Lie(G/H)^H = 0$.
Now, if $\Lie(G/K)^K \neq 0$, then $\Lie(G/K)^H\neq 0$. But $\Lie(G/K)$ is isomorphic to an
$H$-subspace of $\Lie(G/H)$, which would imply that $\Lie(G/H)^H \neq 0$.
Therefore, we conclude that $\Lie(G/K)^K = 0$ and that $|WK|<\infty$.

Now let $x\in V^H$ be a point with isotropy $K$. We have that
\[
 Gx \intersect V^H \iso (G/K)^H
\]
as $N_G H$-spaces. Further, by the preceding lemma, $(G/K)^H$ is a disjoint union of
orbits of the form $N_G H/N_{K'}H$, where $K'$ is conjugate to $K$ and still contains $H$.
Because $|WH|<\infty$, each such orbit is also finite, hence $(G/K)^H$ is a finite
$N_G H$-space.

Therefore, $x$ has a neighborhood in $V^H$ that looks like an open subset of a representation
of $K$, hence all the isotropy subgroups in this neighborhood are subgroups of $K$.
However, we chose $K$ to be minimal, hence all the isotropy subgroups in this neighborhood
are equal to $K$. This shows that $V^K$ has the same dimension as $V^H$, hence $V^K = V^H$.
\end{proof}

Now suppose that $X$ is a $G$-Poincar\'e duality space of dimension $\tau$.
If $x\in X$, write $\tau_x$ for the representation of $G_x$ such that
$\tau(x\colon G/G_x\to X) = G\times_{G_x}\tau_x$.

\begin{condition}\label{nicecondition}
We shall assume that $X$ is embedded in a representation $V$ such that 
\begin{enumerate}
\item for each $x \in X$, $V$ contains a copy of the $G_x$-representation $\tau_x$;
\item for each $x \in X$, $|(V-\tau_x)^G| \geq 2$; and
\item for each $x \in X$, $V-\tau_x$ has gaps of at least two in fixed set dimensions,
i.e., if $H<K\leq G_x$ and $(V-\tau_x)^K \neq (V-\tau_x)^H$, then
$|(V-\tau_x)^K| \leq |(V-\tau_x)^H|-2$.
\end{enumerate}
\end{condition}

We now have the following version of a nonequivariant result.
It generalizes \cite[Theorem 2.2]{CW:spivaknormal} to compact Lie group actions. 

\begin{theorem}\label{spivak}
Assume $X$ is a $G$-Poincar\'e duality space of dimension $\tau$ embedded in a representation $V$
so large that it satisfies Condition~\ref{nicecondition}. Then the Spivak normal fibration $r$ is fiber $G$-homotopy equivalent to a spherical $G$-fibration of dimension $V - \tau$.
\end{theorem}

\begin{proof}
We begin by constructing a candidate $t$ for the Thom class of $r$.
Let $U$ be a regular neighborhood of $X$ as in Definition~\ref{def:spivaknormal}.
 Since $(U,\bndry U)$ is a $G$-manifold of dimension $V$, it has a fundamental class 
 \[
  [U,\bndry U] \in \H^G_V(U,\bndry U;\Mackey A_{G/G})
\]
Take $t \in H_G^{V-\tau}(E,E_0;\Mackey A_{G/G}) \iso H_G^{V-\tau}(U,\bndry U;\Mackey A_{G/G})$ 
to be the Poincar\'e dual of the fundamental class 
$[X] \in \H^G_\tau(X;\Mackey A_{G/G}) \iso \H^G_\tau(U;\Mackey A_{G/G})$, i.e.,
\[
 t \cap [U,\bndry U] = [X].
\]
If $K\leq G$ we now take $K$ fixed-sets throughout and appeal to the nonequivariant
result (given, for example, in \cite[Theorem 9.31]{Ran:surgery}),
to show that the fixed sets of the Spivak normal fibration are each
nonequivariantly equivalent to spherical fibrations, and in particular that
\[
-\cup t^K|_{F_x} \colon H^*(x;\Z) \to \tilde H^{|V^K - \tau^K|+*}(F_x^K;\Z)
\] 
is an isomorphism for each $x \in X^K$, where 
 \[
 F_x = r^{-1}(x) = \Gamma p^{-1}(x)/\Gamma q^{-1}(x). 
 \]
This shows that $F_x^K$ is a homotopy $|V^K-\tau^K|$-sphere for each $x \in X$ and each $K \subset G_x$. 

The result now follows from the following theorem, which shows that $F_x$
is $G_x$-homotopy equivalent to $V-\tau_x$ for each $x$.
\end{proof}

\begin{theorem}\label{thm:sphere}
Let $W$ be a representation of $G$ with $|W^G|\geq 2$ and with gaps of at least two in fixed set dimensions. Assume that $F$ is a $G$-space with each $F^K \hmtpc S^{W^K}$ and that $t$ is a class in $\tilde H_G^W(F; \Mackey A_{G/G})$ such that $t^K \in \tilde H^{W^K}(F^K; \Z)$ is a generator for all $K$. Then $F \hmtpc_G S^W$.
\end{theorem}

\begin{proof}
To start, take a degree one map $S^{W^G}\to F^G$ and extend to a $G$-map
$f\colon S^W\to F$; the obstructions to finding such an extension vanish for dimensional reasons.
By induction up fixed sets, we shall now adjust the degree of $f^K$ to be 1 for all isotropy subgroups $K$ of $W$ for which $WK$ is finite. 

Let $K(W)$ be a $G$-space representing $\tilde H_G^W(-;\Mackey A_{G/G})$ (i.e., the $W$th space of the $G$-spectrum representing $RO(G)$-graded ordinary cohomology \cite{LMM:roghomology}). Restriction to fixed sets in cohomology is represented by a (nonequivariant) map $K(W)^K\to K(\Z;W^K)$. Consider the following diagram.
\[
\xymatrix{
   & A(G) \ar@2{-}[d] \\
  {\pi^G_WF} \ar[r]^-{t_*} \ar@{>->}[d] & {\pi^G_WK(W)} \ar[d]\\
  {\prod_{(K)}\pi_{W^K} F^K} \ar[r]^-{\prod t^K_*} \ar[dr]_{\iso} & {\prod_{(K)} \pi_{W^K}K(W)^K} \ar[d]\\
& {\prod_{(K)}\pi_{W^K}K(\Z,W^K)} \ar@2{-}[r] & {\prod_{(K)} \Z}
  }
\]  
Here, the products are taken over the set of conjugacy classes of subgroups $K$ with $WK$ finite. The fact that the arrow on the left is a monomorphism follows by tom Dieck \cite[Theorem 8.4.1]{tD:reptheory} (which requires the gaps assumed in the statement of our theorem), and similarly for the vertical composite on the right of the diagram. In particular, the image of $\pi^G_W F$ in $\prod_{(K)}\Z$ is a subgroup of the image of $A(G)$, the Burnside ring of $G$. 

For the inductive step, let $K$ be an isotropy subgroup of $S^W$ with $WK$ finite, and assume that we have a $G$-homotopy class $[f] \in \pi^G_W(F)$ with the property that $\deg(f^J) = 1$ for all isotropy subgroups $J > K$ of $S^W$. 
If $L > K$ is not an isotropy group of $S^W$ but has finite Weyl group $WL$, then, 
%since $W$ is admissable, 
by Lemma~\ref{lem:admissable},
$S^{W^L} = S^{W^J}$ for some isotropy subgroup $J > K$ with $WJ$ finite, and so
$\deg (f^L) = 1$.
Since $K$ is an isotropy subgroup of $S^W$, tom Dieck's Theorem~8.4.1 says that $\deg (f^K)$ is completely determined mod $|WK|$ by the numbers $\deg(f^L)$, $L > K$ with $WL$ finite. Further, we can adjust $f$ to realize any degree within this residue class, without affecting the degrees on proper fixed subsets. On the other hand, the images of the two elements $[f]$  and $1 \in A(G)$ in $\prod_{(K)}\pi_{W^K}K(\Z,W^K)$ agree for all $L > K$, and so $\deg(f^K) \congr 1 \mod |WK|$ because the image of $\pi^G_W F$ is a subset of the image of $A(G)$ in $\prod_{(K)}\pi_{W^K}K(\Z,W^K)$. It now follows that we can adjust $f$ so that its degree on $S^{W^K}$ is 1.

This induction leaves us with a map $f\colon S^W \to F$ such that $\deg(f^K) = 1$ for every subgroup $K$ of $G$ with $WK$ finite. To show that $f$ is a $G$-homotopy equivalence, consider the following diagram.
\[
\xymatrix{
 {\tilde H^W_GF} \ar[r]^{f^*} \ar[d] & {\tilde H^W_GS^W} \ar[d] \\
 {\tilde H^{W^K}F^K} \ar[r]^{(f^{K})^*}  & {\tilde H^{W^K}S^{W^K}}
}
\]
We have $f^*(t) = t_*f = 1 \in A(G)$, where $t_*$ is shown in the first diagram in this proof, because elements of $A(G)$ are determined by their fixed set degrees of subgroups $K$ with $WK$ finite. It follows that $(f^K)^*(t^K) = 1$,  and hence $\deg(f^K) = 1$, for every subgroup $K$ of $G$. Thus, $f$ is a $G$-homotopy equivalence.
\end{proof}

We also have the following uniqueness result, which follows by the formal proof in \cite{Br:surgery}. The embedding of $X$ in $V$ determines a collapse map $S^V \to U/\bndry U \hmtpc E/E_0$.
The resulting class $\alpha$ in $[S^V,E/E_0]_G$ satisfies $t\cap \alpha_*[S^V] = [X]$. 

\begin{proposition}
Let $\xi$ be a $(V-\tau)$-dimensional spherical fibration over the G-Poincar\'e duality space $X$, and let $\beta$ be a class in $\pi_V^G(T\xi)$ satisfying $t_\xi \cap \beta_*[S^V] = [X]$. Then there is a fiber homotopy equivalence $b: \xi \to r$, unique up to fiber homotopy, such that $Tb_*(\beta) = \alpha.$
\qed
\end{proposition}

Similar results for a Poincar\'e duality pair $(X,\bndry X)$ can be proved in a like manner.

The first obstruction to a Poincar\'e duality space or pair being $G$-homotopy equivalent to
a manifold is then the obstruction to the Spivak normal fibration being stably spherically
equivalent to a vector bundle. We shall not look at this linearity obstruction further, but
assume for the remainder that it vanishes.

So, now suppose that we have a $(V-\tau)$-dimensional $G$-vector bundle $\xi$ over $X$
and a spherical equivalence $r\to S^{\xi}$. 
As above, this leads to a collapse map $c \colon S^V \to T\xi$.
To proceed further we need the following condition (which is stated in a form that will
be useful for several purposes).

\begin{definition}\label{def:ideal}
Let $\Irr(G)$ denote the set of irreducible representations of $G$.
For $A\in\Irr(G)$, let $\DA_A = \End_G(A)$, the division algebra of $G$-endomorphisms
of $A$, so $\DA_A$ is either $\Real$, $\Cplx$, or $\Qtrn$.
Let $d_A = \dim_\Real \DA_A$, so $d_A$ is either 1, 2, or 4.
If $V$ is any representation of $G$ and $A\in\Irr(G)$, let
\[
 \innprod A V G = \dim_{\DA_A}\Hom_G(A,V),
\]
so the decomposition of $V$ into irreducibles is
\[
 V \iso \Dirsum_{A\in\Irr(G)} A^{\innprod A V G}.
\]
\begin{enumerate}
\item
A representation $V$ of $G$ is said to be {\em ideal} if, for every $K\leq G$, we have
\[
 \innprod \Real V K - \innprod{\Real}{\Lie(G/K)}{K} + 1 
     \leq d_A(\innprod{A}{V}{K} - \innprod{A}{\Lie(G/K)}{K} + 1)
\]
for every $A\in\Irr(K)$ such that $\innprod{A}{V}{K}\neq 0$.

\item
A representation $V$ of $G$ is said to be {\em strongly ideal} if, for every $K\leq G$, we have
\[
 \innprod \Real V K + 1 < d_A(\innprod{A}{V}{K} - \innprod{A}{\Lie(G/K)}{K} + 1)
\]
for every $A\in\Irr(K)$ such that $\innprod{A}{V}{K}\neq 0$.

\item
We say that a representation $\tau$ of $\Pi_G X$ is (strongly) ideal if, for every
$p\colon G/K\to X$, if $\tau(p) = G\times_K \tau_p$, then
$\tau_p$ is (strongly) ideal as a representation of $K$.

\item
We say that a manifold or, more generally, a Poincar\'e duality space, is (strongly) ideal
if the associated (tangent) representation $\tau$ is (strongly) ideal.
\end{enumerate}
\end{definition}

With this definition in hand, let us return to the collapse map $c\colon S^V\to T\xi$ above.
If we suppose that $X$, hence $\tau$, is strongly ideal, then
it follows that the conditions specified in \cite[II.4.13]{Pe:pseudo} are satisfied
and the obstructions to making $c$ transverse to the zero section vanish.
Make $c$ so transverse and let $M = c^{-1}(X)$. We then have a a map $f \colon M \to X$ covered by a map $b \colon \nu \to \xi$, where $\nu$ is the normal bundle to $M$ in $V$.
Thus, we have a trivialization $t\colon TM\dirsum f^*\xi \iso V$. 
Further, as in \cite[\S 3]{CW:spivaknormal} it follows that $f_*[M] = [X]$.
Thus, $(f,\xi,t)$ is a ``degree one normal map'' as in the following definitions.

\begin{definition}\label{def:degreeonemap}
Let $(X,\bndry X)$ and $(Y,\bndry Y)$ be Poincar\'e duality pairs of the same dimension.
A {\em degree one map} $f\colon Y\to X$ is a $G$-map such that
$f_*[Y,\bndry Y] = [X,\bndry X]$.
\end{definition}

\begin{definition}\label{def:normalmap}
Let $M$ be a compact $G$-manifold and $X$ a $G$-space.
A {\em normal map} from $M$ to $X$ is a triple $(f,\xi,t)$ where
$f\colon M\to X$ is a $G$-map, $\xi$ is a vector bundle over $X$,
and $t\colon TM\dirsum f^*\xi \iso U$ is a trivialization, where $U$
is some representation of $G$.
We consider two normal maps equivalent if they differ only by stabilization
of $\xi$ and $t$.
\end{definition}

In sum, every $G$-Poincar\'e duality pair has a spherical Spivak normal fibration, but
there is a possible obstruction to this fibration being linear.
Assuming that this obstruction vanishes,
and assuming that the pair is strongly ideal as above, there is
a degree one normal map from a compact $G$-manifold to the pair.
This is then the starting point for $G$-surgery proper.

%%%%%%%%%%%%%%%%%%%%%%%%%%%%%%%%%%
\section{Homotopy Notions}

In discussing surgery, particularly below the critical dimension, we need to know
what its effect is on homotopy groups. In this section we set up the necessary
machinery and show the effect of surgery in lower dimensions.

The homotopy groups we're concerned about are the nonequivariant homotopy groups of
the components of the fixed sets. As nonequivariantly, we need to consider them as
acted upon by the fundamental group. Equivariantly, it's most convenient to capture that action
by viewing the homotopy groups
as functors on the fundamental groupoid, as follows.

\begin{definition}
If $X$ is a $G$-space and $n\geq 0$, let $\bar\pi_n X$ denote the contravariant functor
on $\Pi_G X$ whose value at $x\colon G/H\to X$ is
\[
 (\bar\pi_n X)(x) = \pi_n(\Map_G(G/H,X),x) = \pi_n(X^H,x) 
\]
where the group on the right (set, if $n=0$) is the usual nonequivariant homotopy group
(writing $x$ as shorthand for $x(eH)$ in that case).
The action of a map $(\alpha,\omega)\colon (y\colon G/K\to X) \to (x\colon G/H\to X)$
is then given by 
\[
 (\alpha,\omega)^*\sigma = \sigma\alpha \gpdact \omega
\]
where $\gpdact$ denotes the usual action of a path on a homotopy element.
(When $n=1$, this can be written as a concatenation of paths as
\[
 (\alpha,\omega)^*\sigma = \omega^{-1} * \sigma\alpha * \omega,
\]
where we concatenate paths from right-to-left as we did in defining $\Pi_G X$.)
\end{definition}

As usual, $\bar\pi_n X$ is set-valued when $n=0$, group-valued when $n=1$, and
abelian group-valued when $n\geq 2$.

A closely related idea is that of the system of covering spaces of $X$.

\begin{definition}
Let $X$ be a $G$-space and suppose that each fixed set $X^H$ is 
locally path connected and semi-locally simply connected
(true, for example, if $X$ is a $G$-CW complex).
Let $X^*$ be the contravariant functor on $\Pi_G X$ whose value at
$x\colon G/H\to X$ is the path component of $X^H$ containing $x(eH)$. Explicitly,
\[
 X^*(x) = X^H_x = \{ p\colon G/H\to X \mid x \hmtpc p \}.
\]
The action of a map $(\alpha,\omega)\colon (y\colon G/K\to X) \to (x\colon G/H\to X)$
is given by
\[
 (\alpha,\omega)^* p = p\alpha.
\]
Note that, if $\lambda\colon x\to p$ is a path in $X^H$, then 
$\lambda\alpha*\omega\colon y\to p\alpha$ is a path in $X^K$.

Let $\tilde X^*$ denote the contravariant functor on $\Pi_G X$ whose
value at $x\colon G/H\to X$ is the simply connected covering space
of $X^H_x$. Explicitly,
\[
 \tilde X^*(x) = \tilde X^H_x = \{ [\lambda]\colon G/H\times I\to X \mid \lambda(-,0) = x \}
\]
where $[\lambda]$ denotes a homotopy class of paths rel endpoints.
As usual, we topologize $\tilde X^*(x)$ so that the endpoint projection
$\tilde X^H_x \to X^H_x$ is a covering space, in fact
the universal covering space.
The action of a map $(\alpha,\omega)\colon (y\colon G/K\to X) \to (x\colon G/H\to X)$
is then given by
\[
 (\alpha,\omega)^*\lambda = \lambda\alpha * \omega.
\]
\end{definition}

Now, let's look more closely at the action of $\Pi_G X$ on $\tilde X^*$.
Fix an $x\colon G/H\to X$ and write $\Aut(x) = \Pi_G X(x,x)$, the automorphism group of
$x$ as an object of $\Pi_G X$. 
(Recall that every self-map in $\Pi_G X$ is an isomorphism, i.e., $\Pi_G X$ is
an {\em EI-category} as in \cite{Lu:transfgroups}.)
From the definitions, we have
\[
 \Aut(x) = \{ (\alpha,\omega) \mid 
 		\alpha\colon G/H\to G/H \text{ and }
		\omega\colon x\to x\alpha \}.
\]
We have an inclusion $\pi_1(X^H,x) \to \Aut(x)$ as a normal subgroup,
given by $\omega \mapsto (\id, \omega)$. The quotient group
is $W_xH$, the subgroup of $WH = NH/H$ that takes $X^H_x$ to itself.
(Here, we identify $\alpha\colon G/H\to G/H$ with $n^{-1}H\in NH/H$ such that
$\alpha(eH) = nH$; we need to take the inverse to make this an isomorphism.)
That is, we have an extension
\[
 1 \to \pi_1 (X^H,x) \to \Aut(x) \to W_xH \to 1.
\]
Note that $W_xH$ includes the identity component of $WH$, so has the same dimension as $WH$.
$\Aut(x)$ is a not-necessarily compact Lie group of that dimension, being the extension
of $W_xH$ by a discrete group.
This extension need not split: Consider, for example,
the free action of $S^1$ on itself by multiplication and take
$x = \id\colon S^1\to S^1$. Then $H = e$, $W_xH = S^1$, 
$\pi_1(S^1,x) \iso \Z$, and $\Aut(x) \iso \Real$, so that the exact sequence above is
\[
 0 \to \Z \to \Real \to S^1 \to 0.
\]

Because $\Aut(x)$ is the automorphism group of $x$, it acts on $\tilde X^*(x) = \tilde X^H_x$.
However, $\tilde X^*$ is a contravariant functor, so we should view this as a right action
and write
\[
 \lambda\cdot (\alpha,\omega) = \lambda\alpha * \omega.
\]
This extends the usual free action of $\pi_1(X^H,x)$ and is compatible with the action
of $W_xH$ on $X^*(x)$, in the sense that the projection
$\tilde X^H_x \to X^H_x$ is $\Aut(x)$-equivariant, where $\Aut(x)$ acts on $X^H_x$
via the projection $\Aut(x)\to W_xH$.
(We think of $WH$ as acting
on the right on $X^H$ via $p\cdot nH = n^{-1}p$.)
This is a principal $(\pi_1(X^H,x);\Aut(x))$-bundle in the language
of \cite{May:equivariantbundles}.

If we take $x$ to be the basepoint in each $\tilde X^*(x)$, the action of $\Pi_G X$ does not
preserves basepoints. However, because each $\tilde X^*(x)$ is simply connected, we can still
take $\pi_n(\tilde X^*(x),x)$ and obtain a well-defined functor $\bar\pi_n\tilde X^*$ on $\Pi_G X$.
Of course, $\bar\pi_0\tilde X^*$ and $\bar\pi_1\tilde X^*$ are trivial. More interestingly, the projection
$\tilde X^*\to X^*$ induces an isomorphism $\bar\pi_n\tilde X^*\to \bar\pi_n X$ for $n\geq 2$.
This is a simple corollary of the nonequivariant result.

We now turn to relative notions and long exact sequences.

\begin{definition}\label{def:homotopypairs}
Let $\phi\colon X\to Y$ be a $G$-map. 
For $x\colon G/H\to X$, let $F(\phi^H,x)$ denote the (nonequivariant) homotopy fiber of the
based map
\[
 (\Map_G(G/H,X),x) \to (\Map_G(G/H,Y),\phi x).
\]
That is,
\[
 F(\phi^H,x) = \{ (p,\lambda) \mid p\colon G/H\to X \text{ and }
					\lambda\colon \phi x\to \phi p \},
\]
where we take actual paths $\lambda\colon G/H\times I\to Y$, not homotopy classes.
We write $x$ again for the point $(x,\overline{\phi x})$, where
$\overline{\phi x}$ denotes the constant path at $\phi x$, and use this as the basepoint.
For a morphism $(\alpha,\omega)\colon (x'\colon G/K\to X) \to (x\colon G/H\to X)$, take a specific
choice of representative $\eta\in\omega$ to get a map
$(\alpha,\eta)^*\colon F(\phi^H,x) \to F(\phi^K,x')$, defined by
\[
 (\alpha,\eta)^*(p,\lambda) = (p\alpha, \lambda\alpha * \phi\eta).
\]
Although this is not a based map, $\eta$ gives a preferred path from
$x'$ to $(\alpha,\eta)^* x$.
For $n\geq 1$, let $\bar\pi_n\phi$ be the
contravariant functor on $\Pi_G X$ whose value at $x$ is
\[
 (\bar\pi_n\phi)(x) = \pi_{n-1} F(\phi^H,x).
\]
The action of a map $(\alpha,\omega)$ is the action of a representative $(\alpha,\eta)$ as above,
using the preferred path determined by $\eta$ to move to the proper basepoint.
One can check that this gives a well-defined functor after passing to homotopy classes.
\end{definition}

We now get the following immediately from the nonequivariant long exact homotopy sequence.
The functor $\phi^*\bar\pi_n Y$ is the pullback of $\bar\pi_n Y$ to $\Pi_G X$.

\begin{proposition}
If $\phi\colon X\to Y$ is a $G$-map, then we have a long exact sequence
of functors on $\Pi_G X$,
\[
 \cdots\to \bar\pi_n X \to \phi^*\bar\pi_n Y\to \bar\pi_n\phi \to \bar\pi_{n-1} X \to \cdots
 \to \bar\pi_1\phi\to \bar\pi_0 X \to \phi^*\bar\pi_0 Y.
\]
\qed
\end{proposition}

Here, by a long exact sequence of functors we mean a sequence that is exact when
evaluated at each object of $\Pi_G X$. It is exact at $\bar\pi_1\phi$ and $\bar\pi_0 X$ in the usual sense.

We will often use the following more explicit description of elements of $\bar\pi_n\phi$. 
If $x\colon G/H\to X$, an element of $(\bar\pi_n\phi)(x)$ can be represented as the homotopy
class of a pair of maps $(\alpha,\beta)$ that fit into a commutative diagram of the following form:
\[
 \xymatrix{
  G/H\times S^{n-1} \ar[d]_i \ar[r]^-\alpha & X \ar[d]^\phi \\
  G/H\times D^n \ar[r]_-\beta & Y
 }
\]
where $\alpha|(G/H\times *) = x$.

We will also need to discuss the homotopy groups of squares, which can be defined
similarly to Definition~\ref{def:homotopypairs} and participate in the expected long exact sequences. 
In practice, we think of them as follows. 
Suppose that $\Phi$ is the following square of $G$-maps:
\[
 \xymatrix{
	A \ar[r]^i \ar[d]_\phi & X \ar[d]^\psi \\
	B \ar[r]_j & Y.
 }
\]
For $n\geq 2$, consider the disc $D^n$ and decompose its boundary sphere as
\[
 \bndry D^n = D_+^{n-1} \union_{S^{n-2}} D_-^{n-1},
\]
where $S^{n-2}$ is the equator, $D_+^{n-1}$ is the upper hemisphere, and $D_-^{n-1}$ is the lower
hemisphere. Taking a basepoint on the equator, an element of 
$\bar\pi_n\Phi(a\colon G/H\to A)$ is the based homotopy
class of a diagram of the following form:
\[
 \xymatrix@!C=4em{
  & G/H\times S^{n-2} \ar[rr]^\alpha \ar[dl] \ar'[d][dd]
    & & A \ar[dd]^\phi \ar[dl]_i \\
  G/H\times D_-^{n-1} \ar[rr]^(.65)\gamma \ar[dd]
    & & X \ar[dd]^(.3)\psi \\
  & G/H\times D_+^{n-1} \ar'[r]_-\beta[rr] \ar[dl]
    & & B \ar[dl]^j \\
  G/H \times D^n \ar[rr]_\delta
    & & Y
 }
\]
Note that we can also write $\bar\pi_n\Phi = \bar\pi_n(\psi,\phi) = \bar\pi_n(j,i)$.

We give a general definition of what we mean by doing surgery on a map
from a manifold to a space.
Later we shall define what we mean by doing surgery on a relative homotopy element.
Here and throughout, we write $|-|$ for the integer dimension of a vector space,
manifold, or orbit of $G$.

\begin{definition}\label{def:surgery1}
Let $M$ be a $G$-manifold, let $X$ be a $G$-space, and let
$f\colon M\to X$ be a $G$-map.
Suppose that we have the following diagram:
\[
 \xymatrix{
   G\times_H(S(V)\times D(W)) \ar[r]^-{\bar\alpha} \ar[d]_i
     & M \ar[d]^f \\
   G\times_H(D(V)\times D(W)) \ar[r]_-{\bar\beta}
     & X
 }
\]
where $V$ and $W$ are representations of $H$ with $|M| = |V|+|W|+|G/H|-1$,
and $\bar\alpha$ is a smooth embedding in the interior of $M$.
Let
\[
 N = (M\times I) \union_{\bar\alpha\times 1} G\times_H (D(V)\times D(W))
\]
with $\bndry N = M\union (\bndry M\times I) \union M'$,
and use projection to $M$ and $\bar\beta$ to extend $f$ to a map $\bar f\colon N\to X$
(well-defined up to homotopy).
Write $f'\colon M'\to X$ for the restriction of $\bar f$ to $M'$.
We call $f\colon M'\to X$ the {\em result of doing surgery on $(\bar\alpha,\bar\beta)$}
and $\bar f\colon N\to X$ the {\em trace} of the surgery.
\end{definition}

Note that this definition is symmetrical in $M$ and $M'$, as there is an embedding
\[
 \bar\alpha'\colon G\times_H (D(V)\times S(W)) \to M'
\]
and we can view $N$ as the trace of a surgery on $(\bar\alpha',\bar\beta)$, with result $M$.

Next, we need to review a bit of algebra. Write $\Pi_G X$-$\Mod$ ambiguously for the category
of contravariant functors on $\Pi_G X$ taking values in sets, groups, or abelian groups.
We mentioned above that a $G$-map $\phi\colon X\to Y$ induces a pullback functor
\[
 \phi^*\colon \text{$\Pi_G Y$-$\Mod$} \to \text{$\Pi_G X$-$\Mod$}.
\]
This functor has a left adjoint
\[
 \phi_!\colon \text{$\Pi_G X$-$\Mod$} \to \text{$\Pi_G Y$-$\Mod$}
\]
that can be defined via a coend:
\[
 (\phi_!T)(y) = \int^{x\in\Pi_G X} \Pi_G Y(y,\phi(x)) \times T(x).
\]
Here, we need to interpret $\Pi_G Y(y,\phi(x)) \times T(x)$ as
\begin{enumerate}
\item
a product of sets, when $T$ is set-valued;

\item
the coproduct $\coprod_{\lambda\in\Pi_G Y(y,\phi(x))} T(x)$ of groups,
when $T$ is group-valued;

\item
the direct sum $\Dirsum_{\lambda\in\Pi_G Y(y,\phi(x))} T(x)$ of abelian groups,
when $T$ is abelian group-valued, which can also be thought of as the tensor product
\[
 \Z\Pi_G Y(y,\phi(x)) \tensor T(x).
\]
\end{enumerate}
One of the most interesting properties of $\phi_!$ is that it preserves represented functors:
\[
 \phi_!\Pi_GX(-,x) = \Pi_GY(-,\phi(x)).
\]

Similarly, for a fixed $X$, consider a subgroup $K\leq G$. There is a functor
\[
 i_K^G\colon \Pi_K X \to \Pi_G X
\]
given by $i_K^G p = G\times_K p$ when $p\colon K/H\to X$, and similarly on morphisms.
This induces the restriction map
\[
 \Res_K^G = (i_K^G)^*\colon \text{$\Pi_G X$-$\Mod$} \to \text{$\Pi_K X$-$\Mod$}
\]
and its left adjoint, the induction map
\[
 \Ind_K^G = (i_K^G)_!\colon \text{$\Pi_K X$-$\Mod$} \to \text{$\Pi_G X$-$\Mod$}.
\]
Induction can again be defined explicitly using a coend:
\[
 (\Ind_K^G T)(p) = \int^{q\in\Pi_K X} \Pi_G X(p,i_K^G q) \times T(q),
\]
and $\Ind_K^G$ preserves represented functors:
\[
 \Ind_K^G\Pi_KX(-,p) = \Pi_GX(-,i_K^Gp).
\]

Returning to surgery, we compare the homotopy groups of the trace of a surgery to those
of the original map.
As in Definition~\ref{def:surgery1},
we suppose that we have done surgery on a diagram
\[
 \xymatrix{
   G\times_H(S(V)\times D(W)) \ar[r]^-{\bar\alpha} \ar[d]_i
     & M \ar[d]^f \\
   G\times_H(D(V)\times D(W)) \ar[r]_-{\bar\beta}
     & X
 }
\]
with trace $\bar f\colon N\to X$. Up to homotopy, we have
\[
 N \hmtpc_G M\union_{\alpha} (G\times_H D(V))
\]
where $\alpha$ is the restriction of $\bar\alpha$ to $G\times_H(S(V)\times 0)$,
and, up to homotopy, $\bar f$ is the map from this pushout to $X$.
We think of this as attaching a $V$-cell to $M$.
We therefore look at the following general situation.

\begin{proposition}\label{prop:killHomotopy}
Let $V$ be a representation of $H\leq G$ with $n = |V^H|\geq 1$.
Let $X$ and $Y$ be $G$-spaces and 
suppose that we have the following diagram of $G$-maps:
\[
 \xymatrix{
   G\times_H S(V) \ar[r]^-{\alpha} \ar[d]_i
     & X \ar[d]^f \\
   G\times_H D(V) \ar[r]_-{\beta}
     & Y
 }
\]
let $X' = X\union_\alpha (G\times_H D(V))$ be the result of attaching
a cell along $\alpha$, let $f'\colon X'\to Y$ be the induced map, and
let $\Phi$ be the induced square:
\[
 \xymatrix{
	X \ar[r]^\phi \ar[d]_f & X' \ar[d]^{f'} \\
	Y \ar@{=}[r] & Y
 }
\]
Then $\phi_*\colon \Pi_G X\to \Pi_G X'$ is essentially surjective over $\orb G$ and
$\bar\pi_0 X\to \phi^*\bar\pi_0 X'$ is (objectwise) surjective. Further,
\begin{enumerate}
\item
if $n=1$, then $\bar\pi_1 f\to \phi^*\bar\pi_1 f'$ is surjective and the following is
an exact sequence of functors on $\Pi_G X$ in the usual sense, where the first
map is induced by $(\alpha,\beta)$:
\[
 \alpha_!\Ind_H^G\bar\pi_1(D(V),S(V)) \to \bar\pi_1 f \to \phi^*\bar\pi_1 f';
\]

\item
if $n=2$, then $\bar\pi_2\Phi = 0$, $\bar\pi_1 f\iso \phi^*\bar\pi_1 f'$, and the following
is an exact sequence of functors on $\Pi_G X$, where the first map is induced by $(\alpha,\beta)$:
\[
 \alpha_!\Ind_H^G\bar\pi_2(D(V),S(V)) \to \bar\pi_2 f \to \phi^*\bar\pi_2 f' \to 0;
\]
and

\item
if $n\geq 3$, then
\begin{enumerate}
\item $\Pi_G X'$ is equivalent to $\Pi_G X$ over $\orb G$,
\item $\bar\pi_k\Phi = 0$ for $2\leq k \leq n$, $\bar\pi_k f \iso \bar\pi_k f'$ for $1\leq k \leq n-1$, and
\item the following
is an exact sequence of functors on $\Pi_G X$, where the first map is induced by $(\alpha,\beta)$:
\[
 \alpha_!\Ind_H^G\bar\pi_{n}(D(V),S(V)) \to \bar\pi_{n} f \to \bar\pi_{n} f' \to 0.
\]
\end{enumerate}
\end{enumerate}
\end{proposition}

\begin{proof}
The statements about $\Pi_G X\to \Pi_G X'$ are clear because $V^H\neq 0$, and justify
considering $\bar\pi_k X'$ to be a functor on $\Pi_G X$ in part (3) of the statement.
The surjectivity of $\bar\pi_0 X\to \phi^*\bar\pi_0 X'$
and of $\bar\pi_1 f\to \bar\pi_1 f'$ for all $n$ are also clear.

The claimed isomorphisms $\bar\pi_k f \iso \bar\pi_k f'$ follow from the claimed vanishings
of $\bar\pi_k\Phi$ and the long exact sequence.
There is another long exact sequence in which $\bar\pi_*\Phi$ participates, in which the
other two terms are $\bar\pi_*\phi$ and $\bar\pi_*(1_Y) = 0$. From this it follows that
$\bar\pi_k\Phi \iso \bar\pi_{k-1}\phi$, and we show that the latter vanishes in the range claimed.

So, consider $\bar\pi_k\phi$ for $1 \leq k \leq n-1$.
A typical element of $(\bar\pi_k\phi)(x\colon G/K\to X)$ can be thought of, in adjoint form,
as the homotopy class of a nonequivariant diagram of based spaces (taking the north pole as the basepoint of $S^{k-1}$) of the following form:
\[
 \xymatrix{
  S^{k-1} \ar[d] \ar[r]^-{\gamma} & (X^K,x) \ar[d]^{\phi} \\
  D^k \ar[r]_-{\delta} & ([X\union_\alpha (G\times_H D(V))]^K,x)
 }
\]
Such a diagram will represent the trivial element if $\delta$ fails to meet the center of the disc,
$(G/H\times 0)^K$. This is a disjoint union of $WK$-orbits, associated with conjugates of $K$ lying in $H$.
Letting $K$ act on $V$ through such a conjugate, the codimension of the correponding orbit in the disc is $|V^K| \geq |V^H| > k$, so, for dimensional reasons, $\delta$ is homotopic to a map missing
the center of the disc, and all such diagrams are trivial.
Therefore, $\bar\pi_k\phi = 0$ for $1\leq k \leq n-1$ as claimed. 

From the long exact homotopy sequence, we now get an exact sequence
\[
 \bar\pi_{n+1}\Phi \to \bar\pi_{n} f\to \phi^*\bar\pi_{n} f' \to 0,
\]
which we can write as
\[
 \bar\pi_{n}\phi \to \bar\pi_{n} f\to \phi^*\bar\pi_{n} f' \to 0,
\]
Consider the map
$\alpha_!\Ind_H^G\bar\pi_{n}(D(V),S(V)) \to \bar\pi_{n}\phi$ induced by $(\alpha,\beta)$
Suppose that we have an element of $(\bar\pi_{n}\phi)(x\colon G/K\to X)$, represented
by the following nonequivariant diagram:
\[
 \xymatrix{
  S^{n-1} \ar[d] \ar[r]^-{\gamma} & (X^K,x) \ar[d]^{\phi} \\
  D^{n} \ar[r]_-{\delta} & ([X\union_\alpha (G\times_H D(V))]^K,x)
 }
\]
If $K$ acts on $V$ through a conjugate contained in $H$,
we will have $|V^K|\geq |V^H| = n$, so $\delta$ may well hit the center of
the disc. If it does, for dimensional reasons we may assume (after homotopy) that it does so at
a finite number of isolated points in the interior of $D^{n}$,
at each of which it meets $(G/H\times 0)^K$ transversely.
We can then homotope $\delta$ so that a small ball around each of these points is mapped
homeomorphically to $gH\times D(V^K)$ for some $g\in G$, while the rest of $D^{n}$ is
mapped to the boundary sphere $\alpha(G/H\times S(V^H))$ except for a tail connecting $x(eK)$ to
the sphere.
This exactly describes a sum of elements coming from $\alpha_!\Ind_H^G \bar\pi_{n}(D(V),S(V))$,
showing that 
\[
 \alpha_!\Ind_H^G \bar\pi_{n}(D(V),S(V))
  \to \bar\pi_{n}\phi
\]
is onto, hence
\[
 \alpha_!\Ind_H^G\bar\pi_{n}(D(V),S(V)) \to \bar\pi_{n-1} X \to \phi^*\bar\pi_{n-1} X' \to 0
\]
is exact as claimed.
\end{proof}

\begin{corollary}\label{cor:surgeryKillsHomotopy}
Let $M$ be a $G$-manifold, let $X$ be a $G$-space, and let $f\colon M\to X$
be a $G$-map.
Let $V$ be a representation of $H\leq G$ with $n = |V^H|\geq 1$.
Consider the following diagram:
\[
 \xymatrix{
   G\times_H(S(V)\times D(W)) \ar[r]^-{\bar\alpha} \ar[d]_i
     & M \ar[d]^f \\
   G\times_H(D(V)\times D(W)) \ar[r]_-{\bar\beta}
     & X
 }
\]
where $\bar\alpha$ is an embedding
with $|M| = |V| + |W| + |G/H| - 1$, and let
$f'\colon M'\to X$ be the result of doing surgery on $(\bar\alpha,\bar\beta)$.
Let $m = |V^H| + |W^H| + |(G/H)^H| - 1$, so $m$ is the dimension of the components
of $M^H$ containing the image of $\alpha^H$.
Let $w = |WH| = |(G/H)^H|$.
If $2n < m - w + 1$, which we can also write (for later purposes) as
\[
 n \leq \begin{cases}
			\lfloor (m-w+1)/2 \rfloor & \text{if $m-w$ is even} \\
			\lfloor (m-w+1)/2 \rfloor - 1 & \text{if $m-w$ is odd,}
		\end{cases}
\]
then
\begin{enumerate}
\item
$\bar\pi_k f \iso \bar\pi_k f'$ for $1\leq k \leq n - 1$ and

\item
there is an exact sequence
\[
 \alpha_!\Ind_H^G\bar\pi_{n} (D(V),S(V)) \to \bar\pi_{n} f \to \bar\pi_{n} f' \to 0.
\]
\end{enumerate}
\end{corollary}

\begin{proof}
Let $\bar f\colon N\to X$ be the trace of the surgery. 
As we pointed out earlier, we have
\[
 N \hmtpc_G M\union (G\times_H D(V))
\]
and also
\[
 N \hmtpc_G M' \union (G\times_H D(W)).
\]
From the first equivalence and Proposition~\ref{prop:killHomotopy}, we see that
$\bar\pi_k \bar f \iso \bar\pi_k f$ for $1\leq k \leq n-1$ and that we have an exact sequence
\[
\alpha_!\Ind_H^G\bar\pi_{n} (D(V),S(V)) \to \bar\pi_{n} f \to \bar\pi_{n} \bar f \to 0.
\]
Now consider the second equivalence and let $p = |W^H|$. 
Our assumptions imply that $p > n$,
so by Proposition~\ref{prop:killHomotopy} again, we have
$\bar\pi_k \bar f \iso \bar\pi_k f'$ for $1\leq k \leq n$.
Substituting $\bar\pi_k f'$ for $\bar\pi_k \bar f$
and $\bar\pi_{n}f'$ for $\bar\pi_{n}\bar f$ above gives the claims of the corollary.
\end{proof}

We interpret this result as follows: Fix a basepoint in $S(V)^H$,
so the group $\pi_{n}(D(V)^H,S(V)^H)$ is generated by the identity 
$(D^n,S^{n-1}) \xrightarrow{\iso} (D(V)^H,S(V)^H)$.
Composing with the adjoint maps $(\hat\alpha,\hat\beta)$ 
gives us a map $(D^n,S^{n-1})\to f^H$, hence a homotopy element
$[\hat\alpha,\hat\beta]\in\pi_{n}(f^H,x)$ where $x$ is the image of the basepoint.
The effect of doing surgery on $(\alpha,\beta)$ is then to leave the homotopy groups below
dimension $n$ unchanged and, in dimension $n$, to kill not just $[\hat\alpha,\hat\beta]$, but the
whole $\Pi_G X$-module it generates.

Concentrating on just $\pi_{n}(f^H,x)$, recall that we have an action of $\Aut(x)$
on this group. Here, the effect of surgery is to kill the $\Aut(x)$-module generated
by $[\hat\alpha,\hat\beta]$. 
This is the generalization of the nonequivariant result \cite[IV.1.5]{Br:surgery},
that surgery kills the $\pi_1 X$-module generated by an element.

Having discussed the homotopy groups of the fixed sets,
we also need a brief discussion of their homology groups.
We will be interested in homology with ``universal local coefficients,''
and there are two ways of looking at this. One is to consider the homology of
the universal covering spaces.
Let $X$ be a $G$-space such that each $X^H$ is locally path connected and
semi-locally simply connected. Recall that we then have the contravariant
functor $\tilde X^*$ on $\Pi_G X$ whose value $\tilde X^*(x)$ at $x\colon G/H\to X$
is the simply-connected covering space of $X^H$ based at $x$.
Applying homology with $\Z$ coefficients, we get the contravariant
functor $H_*(\tilde X^*;\Z)$.
Because $\Aut(x)$ acts on $\tilde X^*(x)$, we get an action of
$\pi_0\Aut(x)$ on $H_*(\tilde X^*(x);\Z)$, so we can think 
of the homology groups as $\Z\pi_0\Aut(x)$-module.
This is the generalization of the nonequivariant view of
$H_*(\tilde X;\Z)$ as a $\Z\pi_1(X)$-module, but note that
$\pi_0\Aut(x)$ is not $\pi_1(X^H_x,x)$; it contains a quotient
of $\pi_1(X^H_x,x)$ and will be larger than that quotient
if $W_xH$ is not connected.

Alternatively, we can consider the local coefficient system
$\Z\Pi X^H(x,-)$ on $X^H_x$.
Nonequivariantly, we know that we have the isomorphism
\[
 H_*(X^H_x;\Z\Pi X^H(x,-)) \iso H_*(\tilde X^*(x);\Z).
\]
(However we define the left-hand side, the isomorphism
is a consequence of the Serre spectral sequence for
the covering $\tilde X\to X$.)
We can view $H_*(X^H_x;\Z\Pi X^H(x,-))$ as a contravariant functor of $x$
as follows:
Recall that the action of $(\alpha,\omega)\colon (y\colon G/K\to X) \to (x\colon G/H\to X)$
on $X^*$ is given by $(\alpha,\omega)^*(p\colon G/H\to X) = p\alpha$.
We then define $(\alpha,\omega)^*\colon \Pi X^H(x,-) \to \Pi X^K(y,-)$ on a path
$\lambda\colon G/H\times I\to X$ starting at $x$ by
\[
 (\alpha,\omega)^*\lambda = \lambda\alpha*\omega,
\]
a by-now familiar formula. It's now easy to check that the
map
\[
 (\alpha,\omega)^*\colon \Z\Pi X^H(x,-) \to \Z\Pi X^K(y,-)
\]
is a map of local coefficient systems over the map $(\alpha,\omega)^*\colon X^H_x\to X^K_y$.
Taking the induced map on homology makes $H_*(X^H_x;\Z\Pi X^H(x,-))$
a contravariant functor on $\Pi_G X$
isomorphic to the functor $H_*(\tilde X^*;\Z)$ above.

We memorialize this discussion and introduce notation in the following definition.

\begin{definition}\label{def:fixedsethomology}
Let $X$ be a $G$-space. When we write $\Mackey H_*(X^*)$ or $H_*(X^H_x)$ without coefficients, we shall mean
the contravariant functor on $\Pi_G X$ defined by
\[
 \Mackey H_*(X^*)(x) = H_*(X^H_x;\Z\Pi X^H_x(x,-)) \iso H^*(\tilde X^*(x);\Z).
\]
Similarly, when we write $\MackeyOp H^*(X^*)$ or $H^*(X^H_x)$, we shall mean the covariant functor defined by
\[
 \MackeyOp H^*(X^*)(x) = H^*(X^H_x;\Z\Pi X^H(-,x)) \iso H^*(\tilde X^*(x);\Z).
\]
We call these the {\em fixed-set homology and cohomology groups of $X$.}
We use similar definitions for relative homology and cohomology, and homology and cohomology
of squares.
\end{definition}

Notice that the nonequivariant Hurewicz map induces a natural transformation
\[
 \Mackey \pi_k(X) \to \Mackey H_k(X^*)
\]
for $k\geq 2$. The easiest way to see this is to note that $\pi_k(X^H_x,x) \iso \pi_k(\tilde X^H_x)$
for $k\geq 2$ and then apply the Hurewicz map to $\tilde X^H_x$.
(There are similar maps available for $k=0$ and 1, but we should not need them.)
In particular, note that, for any $x\colon G/H\to X$ and $k\geq 2$,
\[
 \pi_k(X^H_x,x) \to H_k(X^H_x;\Z\Pi X(x,-))
\]
is a map of $\Z\pi_0\Aut(x)$-modules.

We can now state the following criteria for a map to be a homotopy equivalence.
We use the notation from Definition~\ref{def:universalcoeffs} as well as
the notation above.

\begin{theorem}\label{thm:homotopyequivalence}
Let $f\colon X\to Y$ be a $G$-map between spaces of the homotopy types of $G$-CW complexes.
Suppose that $f_*\colon \Pi_G X\to \Pi_G Y$ is an equivalence over $\orb G$.
Then the following are equivalent:
\begin{enumerate}
\item $f\colon X\to Y$ is a $G$-homotopy equivalence.
\item $f_*\colon \Mackey H^G_*(X)\to \Mackey H^G_*(Y)$ is an isomorphism in integer grading.
\item $f_*\colon \Mackey H_*(X^*)\to \Mackey H_*(Y^*)$ is an isomorphism (in integer grading).
\end{enumerate}
\end{theorem}

\begin{proof}
(1) clearly implies both (2) and (3). 

(2) implies (3) by universal coefficients and the Atiyah-Hirzebruch spectral sequence.

Assuming (3), we have, for every subgroup $H$, that
$f^H\colon X^H\to Y^H$ is, nonequivariantly, an equivalence on fundamental groupoids
and induces an isomorphism on homology with local $\Z\pi_1 X$-coefficients.
Therefore, by the Hurewicz isomorphism for pairs, $f^H$ is a nonequivariant homotopy equivalence
for all $H$, which implies (1).
\end{proof}

%%%%%%%%%%%%%%%%%%%%%%%%%%%%%%%%%
\section{The Surgery Step}

We now define what we mean by surgery on a homotopy element
and examine under what circumstances such surgery is possible.

\begin{definition}\label{def:surgery2}
Let $M$ be a compact $G$-manifold, let $X$ be a $G$-space, and
let $(f,\xi,t)\colon M\to X$ be a normal map (as in Definition~\ref{def:normalmap}).
Let $[\alpha,\beta]\in \bar\pi_{n} f$ be represented by the following diagram:
\[
 \xymatrix{
  G/H\times S^{n-1} \ar[d] \ar[r]^-\alpha & M \ar[d]^f \\
  G/H\times D^n \ar[r]_-\beta & X
 }
\]
To say that {\em we can do surgery on $[\alpha,\beta]$} is to say
first that $\alpha$ is homotopic to a map that can be thickened to an embedding
\[
 \bar\alpha\colon G\times_H(S^{n-1}\times D(W))\to M
\]
in the interior of $M$, for some representation $W$ of $H$
with $|M| = n + |W| + |G/H| - 1$.
Using the contractibility of $D(W)$, we extend $\beta$ to get the following diagram:
\[
 \xymatrix{
  G\times_H (S^{n-1}\times D(W)) \ar[d] \ar[r]^-{\bar\alpha} & M \ar[d]^f \\
  G\times_H (D^n\times D(W)) \ar[r]_-{\bar\beta} & X
 }
\]
We then do surgery on $(\bar\alpha,\bar\beta)$ in the sense of Definition~\ref{def:surgery1}, 
with $f'\colon M'\to X$ being the result of
the surgery and $\bar f\colon N\to X$ its trace.
Orienting $I$ upwards, we have the trivialization
\[
 t\dirsum\Real \colon T(M\times I)\dirsum f^*\xi
   \iso TM\dirsum\Real\dirsum f^*\xi
   \xrightarrow{\iso} U\dirsum\Real.
\]
We then require that we can extend $t\dirsum\Real$ over all of $N$
(possibly after stabilizing further) to get a trivialization
\[
 \bar t\colon TN\dirsum \bar f^*\xi \xrightarrow{\iso} U\dirsum\Real.
\]
Restricting $\bar t$ to the boundary and using the outward normal along $M'$,
we get a trivialization
\[
 t'\colon TM'\dirsum (f')^*(\xi\dirsum\Real) \iso U\dirsum\Real,
\]
hence a normal map $(f',\xi\dirsum\Real,t')\colon M'\to X$.
\end{definition}

The question we address in this section is: Given an element of $\bar\pi_n f$,
under what conditions can we do surgery on it?
We first need some results on destabilization of bundle maps.

\begin{proposition}
Let $\DA = \Real$, $\Cplx$, or $\Qtrn$, and let $d = \dim_\Real \DA$.
Let $Z$ be a nonequivariant CW complex with $n$-dimensional $\DA$-bundles $\zeta$ and $\xi$ and
a bundle isomorphism $f\colon \zeta\dirsum\DA^k\to\xi\dirsum\DA^k$ for some $k\geq 0$.
If $\dim Z < d(n+1) - 1$, then there exists a bundle isomorphism $f'\colon \zeta\to\xi$ such that
$f \hmtpc f'\dirsum\DA^k$.
\end{proposition}

\begin{proof}
Write $\Gamma(m)$ for $O(m)$, $U(m)$, or $Sp(m)$, if $\DA = \Real$, $\Cplx$, or $\Qtrn$,
respectively.
For any $m$,
the fibration $B\Gamma(m) \to B\Gamma(m+1)$ has fiber $S^{d(m+1)-1}$, so is $[d(m+1)-1]$-connected.
Hence the map $B\Gamma(n) \to B\Gamma(n+k)$ is $[d(n+1)-1]$-connected.
It follows from the Whitehead theorem that
$[Z,B\Gamma(n)] \iso [Z,B\Gamma(n+k)]$ if $\dim Z < d(n+1)-1$, from which the proposition follows.
\end{proof}

In the following we use again the notation from Definition~\ref{def:ideal}.

\begin{corollary}\label{cor:destab}
Suppose that $Z$ is a nonequivariant CW complex considered as a $G$-space with trivial action.
Let $\xi$ be a $G$-vector bundle over $Z$, let $V$ and $W$ be representations of $G$,
and suppose that we have a trivialization $c\colon \xi\dirsum W\to V\dirsum W$.
If
\[
 \dim Z < d_A(\innprod{A}{V}{G} + 1) - 1
\]
for all $A\in\Irr(G)$ such that $\innprod{A}{V}{G} > 0$ and
$\innprod{A}{W}{G} > 0$, then there exists a trivialization
$c'\colon \xi\to V$ such that $c\hmtpc c'\dirsum W$.
\end{corollary}

\begin{proof}
Because $G$ acts trivially on $Z$, we can decompose $\xi$ as
\[
 \xi \iso \Dirsum_{A\in\Irr(G)} A\tensor_{\DA_A} \Hom_G(A,\xi).
\]
Writing $c_A = \Hom_G(A,c)$, we get
\[
 c_A\colon \Hom_G(A,\xi) \dirsum \DA_A^{\innprod{A}{W}{G}}
  \xrightarrow{\iso} \DA_A^{\innprod{A}{V}{G}} \dirsum \DA_A^{\innprod{A}{W}{G}}.
\]
By the preceding proposition and our assumption on the dimension of $Z$, there is a bundle map
\[
 c'_A\colon \Hom_G(A,\xi) \xrightarrow{\iso} \DA_A^{\innprod{A}{V}{G}}
\]
such that $c_A \hmtpc c'_A \dirsum \DA_A^{\innprod{A}{W}{G}}$.
(The cases where either $\innprod{A}{V}{G} = 0$ or
$\innprod{A}{W}{G} = 0$ are trivial.)
Tensoring with $A$ and summing over all $A$, we get the $c'$ we claimed.
\end{proof}

This then allows us to prove the following lemma.
We use the notion of ``ideal'' defined in Definition~\ref{def:ideal}.
In the following we use the notation $V_H$ for the orthogonal complement of
$V^H$ in $V$.

\begin{lemma}\label{lem:trivialnormal}
Let $M$ be an ideal $G$-manifold, let $\alpha\colon G/H\times S^{n-1}\to M$ be a $G$-map,
and let $\hat\alpha\colon S^{n-1}\to M^H$ be the (nonequivariant) adjoint map.
Suppose that the image of $\hat\alpha$ lies in the space of points with isotropy exactly $H$.
Write $M^H_{\alpha}$ for the component of $M^H$ containing the image of $\hat\alpha$.
Suppose that we have a stable $H$-trivialization
\[
 c\colon \hat\alpha^*TM\dirsum W \xrightarrow{\iso} V\dirsum W.
\]
If
\[
 n \leq |M^H_{\alpha}| - |WH|,
\]
then there exists a trivialization
\[
 c'\colon (\hat\alpha^*TM)_H - \Lie(G/NH) \xrightarrow{\iso} V_H - \Lie(G/NH)
\]
such that $c_H \hmtpc c'\dirsum \Lie(G/NH)\dirsum W_H$.
\end{lemma}

\begin{proof}
The map $\alpha$ induces a bundle monomorphism
$S^{n-1}\times\Lie(G/H) \includesin \hat\alpha^*TM$, by our assumption that all points in the image
of $\hat\alpha$ have isotropy $H$.
Because $\Lie(G/H)_H = \Lie(G/NH)$, we get that
$S^{n-1}\times\Lie(G/H)$ is a subbundle of $(\hat\alpha^*TM)_H$, which we used
in stating the conclusion of the lemma.
Using the isomorphism $c$, restricted to any fiber, we see that $\Lie(G/NH)$ must also
be a subrepresentation of $V$.

The trivialization $c$ restricts to a bundle isomorphism
\begin{multline*}
 c_H\colon ((\hat\alpha^*TM)_H - \Lie(G/NH))\dirsum (\Lie(G/NH)\dirsum W_H) \\
   \xrightarrow{\iso}
   (V_H - \Lie(G/NH)) \dirsum (\Lie(G/NH)\dirsum W_H).
\end{multline*}
We now want to apply Corollary~\ref{cor:destab}.
Let $x$ be a point in the image of $\hat\alpha$ and let $\tau_x$ be the
tangent $H$-representation at $x$.
By assumption, and using that $|WH| = |\Lie(G/H)^H|$, we have
\[
 n - 1 < \innprod{\Real}{\tau_x}{H} - \innprod{\Real}{\Lie(G/H)}{H}
  \leq d_A(\innprod{A}{\tau_x}{H} - \innprod{A}{\Lie(G/H)}{H} + 1) - 1
\]
for $A\neq\Real$. Therefore, with $Z = S^{n-1}$, we can apply
Corollary~\ref{cor:destab} to obtain the conclusion of this lemma.
\end{proof}

Now we can state and prove the main result of this section.

\begin{theorem}\label{thm:surgerypossible}
Let $M$ be an ideal $G$-manifold, let $X$ be a $G$-space, and
suppose that $(f,\xi,t)\colon M\to X$ is a normal map.
Let $[\alpha,\beta]\in \bar\pi_n(f)(x\colon G/H\to M)$,
so $\alpha$ and $\beta$ fit in the following diagram:
\[
 \xymatrix{
  G/H\times S^{n-1} \ar[d]_i \ar[r]^-\alpha & M \ar[d]^f \\
  G/H\times D^n \ar[r]_-\beta & X
 }
\]
Let $M^H_x$ denote the component of $M^H$
containing $x(eH)$
and let $W_x H$ denote the subgroup of $WH$ carrying $M^H_x$ to itself.
Let $m = |M^H_x|$ and let $w = |W_x H| = |WH|$.
Suppose that
\[
 n \leq m - w - 1
\]
and
\[
 n \leq m - w - |(M^H_x)^K|
\]
for every $K$ strictly containing $H$.
Then
\begin{enumerate}
\item
$\alpha$ is homotopic to a map that can be thickened to an immersion
\[
 \gamma\colon G\times_H(S^{n-1}\times D^{m-w-n+1} \times D(W)) \to M
\]
for some $H$-representation $W$ with $W^H=0$, and $\gamma$ is unique up to regular homotopy;
and

\item
if $\gamma$ can be taken to be an embedding, then we can do surgery on $[\alpha,\beta]$.
\end{enumerate}
\end{theorem} 

\begin{proof}
We start with the first statement.
Consider the following adjoint diagram of nonequivariant maps:
\[
 \xymatrix{
  S^{n-1} \ar[d]_i \ar[r]^-{\hat\alpha} & M^H_x \ar[d]^{f^H}\\
  D^n \ar[r]_-{\hat\beta} & X^H
 }
\]
If $K$ strictly contains $H$, then the $W_x H$-orbit of $(M^H_x)^K$
has dimension no more than $w+|(M^H_x)^K| < m - (n - 1)$.
Because $\hat\alpha$ can meet only finitely many orbit types in $M^H_x$,
it now follows for dimensional reasons that $\hat\alpha$ is homotopic to a map
that misses all fixed points by subgroups strictly larger than $H$.
Hence, we may assume that the image of $\hat\alpha$ lies in $\breve M^{H}_x$, the open submanifold
of $M^H_x$ consisting of points with isotropy exactly $H$.

The trivialization $t$ pulls back to a trivialization
\[
 \alpha^*t\colon \alpha^*TM \dirsum \alpha^*f^*\xi \xrightarrow{\iso} G\times_H U.
\]
Because $D^n$ is contractible, we also have a trivialization
\[
 \alpha^*f^*\xi = i^*\beta^*\xi \xrightarrow{\iso} G\times_H V
\]
for some $H$-representation $V$, so we have
\[
 \alpha^*TM\dirsum (G\times_H V) \xrightarrow{\iso} G\times_H U.
\]
Note that this implies that $U$ contains a copy of $V$ as a subrepresentation.

Restrict to $S^{n-1}$ and take $H$-fixed points, to get
\[
 \hat\alpha^* TM^H \dirsum V^H \xrightarrow{\iso} U^H.
\]
Because $W_x H$ acts freely on $\breve M^H_x$,
the quotient $\breve M^H_x/W_x H$ is again a manifold, and the projection
$q\colon \breve M^H_x \to \breve M^H_x/W_x H$
induces a splitting
\[
 T\breve M^H_x \iso q^*T(\breve M^H_x/W_x H) \dirsum \Lie(WH).
\]
Hence, we have an isomorphism
\[
 \hat\alpha^*q^*T(\breve M^H_x/W_x H) \dirsum \Lie(WH) \dirsum V^H \xrightarrow{\iso} U^H.
\]

On the other hand, consider the nonequivariant manifolds
$S^{n-1}\times D^{m-w-n+1}$ and $D^n\times D^{m-w-n+1}$. 
Because $m-w-n+1\geq 1$, we can find trivializations
\[
 T(S^{n-1}\times D^{m-w-n+1}) \xrightarrow{\iso} \Real^{m-w}
\]
and
\[
 T(D^n\times D^{m-w-n+1}) \xrightarrow{\iso} \Real^{m-w+1}
\]
that are compatible, in the sense that the restriction of the second
to $ S^{n-1}\times D^{m-w-n+1}$ is homotopic to the first
plus addition of the inward-pointing normal.
Comparing dimensions, we can choose a nonequivariant isomorphism
\[
 \Real^{m-w}\dirsum\Lie(WH)\dirsum V^H \iso U^H.
\]
This gives us compatible isomorphisms
\[
 T(S^{n-1}\times D^{m-w-n+1})\dirsum\Lie(WH)\dirsum V^H 
  \xrightarrow{\iso} \Real^{m-w}\dirsum\Lie(WH)\dirsum V^H
  \xrightarrow{\iso} U^H
\]
and
\begin{multline*}
 T(D^n\times D^{m-w-n+1})\dirsum\Lie(WH)\dirsum V^H \\
  \xrightarrow{\iso} \Real^{m-w+1}\dirsum\Lie(WH)\dirsum V^H
  \xrightarrow{\iso} U^H\dirsum\Real.
\end{multline*}

Putting together two of the isomorphisms above, we get
\begin{multline*}
 T(S^{n-1}\times D^{m-w-n+1})\dirsum\Lie(WH)\dirsum V^H \\
  \xrightarrow{\iso}
  \hat\alpha^*q^*T(\breve M^H_x/W_x H) \dirsum \Lie(WH) \dirsum V^H.
\end{multline*}
As in \cite{Wal:surgery}, and using the assumption that
$n-1 \leq m-w-2$,
we can now appeal to \cite{HirMW:Immersions}
to say that $q\hat\alpha$ is homotopic to a map that extends to an immersion
\[
 \bar\gamma\colon S^{n-1}\times D^{m-w-n+1}
  \to \breve M^H_x/W_x H
\]
whose derivative is stably homotopic to the isomorphism just displayed above,
and that $\bar\gamma$ is determined up to regular homotopy by this property.
We can now lift the homotopy of $q\hat\alpha$ to a homotopy of $\hat\alpha$
and use $\bar\gamma$ to thicken in the directions normal to the orbits of $W_x H$,
to get a $W_x H$-immersion
\[
 \gamma_0\colon W_x H\times S^{n-1}\times D^{m-w-n+1} \to M^H_x
\]
whose quotient is $\bar\gamma$ and whose restriction to $W_x H\times S^{n-1}\times 0$
is homotopic to the original $\alpha|(W_x H\times S^{n-1})$.

We've now thickened in the $H$-trivial directions, but it remains to deal
with the non-trivial directions, normal to $M^H_x$.
Our assumptions allow us to apply Lemma~\ref{lem:trivialnormal} to the trivialization
\[
 \hat\alpha^* TM \dirsum V \xrightarrow{\iso} U
\]
to get a trivialization
\[
 (\hat\alpha^*TM)_H - \Lie(G/NH) \xrightarrow{\iso} U_H - V_H - \Lie(G/NH)
\]
stably homotopic to $\hat\alpha^* t_H$.
Let $W = U_H - V_H - \Lie(G/NH)$, so $W^H = 0$.
We can use the trivialization just obtained to thicken $\gamma_0$ in the direction
of $(\hat\alpha^*TM)_H - \Lie(G/NH)$, then extend the action to get the immersion
we claimed:
\[
 \gamma\colon G\times_H(S^{n-1}\times D^{m-w-n+1}\times D(W)) \to M.
\]
The uniqueness of $\gamma$ follows from the uniqueness of $\bar\gamma$.

We now turn to the second statement of the theorem, so assume that $\gamma$ is
an embedding.
Form the trace $N$ using $\gamma$ and extend $f$ to a map $\bar f\colon N\to X$
using $\beta$. We need to show that we can extend the trivialization $t\dirsum\Real$
on $M\times I$ to all of $N$, which is to say, we need to extend it over the handle.
It suffices to extend the $H$-trivialization over
$e\times D^{n}\times D^{m-w-n+1}\times D(W)$.
On this space we have the trivialization
\[
 T(D^{n}\times D^{m-w-n+1})\dirsum\Lie(G/H)\dirsum W\dirsum \hat\beta^*\xi
  \xrightarrow{\iso} U\dirsum\Real
\]
coming from its contractibility. The restriction of this isomorphism is homotopic
to $\Real$ added to a map
\[
 T(S^{n-1}\times D^{m-w-n+1})\dirsum\Lie(G/H)\dirsum W\dirsum \hat\alpha^*f^*\xi
  \xrightarrow{\iso} U.
\]
On the other hand, $\gamma$ was chosen so that this is the same trivialization
induced by $t$.
(This is not meant to be obvious, but can be checked carefully from the definitions
of all of the maps involved.)
Therefore, we can extend the trivialization over the handle
as required.
\end{proof}

\begin{corollary}\label{cor:surgerypossible}
Under the assumptions of Theorem~\ref{thm:surgerypossible}, if we also have
\[
 n \leq \lfloor (m-w+1)/2 \rfloor,
\]
then we can do surgery on $[\alpha,\beta]$.
\end{corollary}

\begin{proof}
The condition on $n$ is equivalent to $2(n-1) < m - w$. In that case,
the immersion 
$\bar\gamma\colon S^{n-1}\times D^{m-w-n+1} \to \breve M^H_x/W_x H$
constructed in the proof of the theorem is regularly homotopic to a
map that is an embedding on $S^{n-1}\times 0$. By shrinking the disk $D^{m-w-n+1}$
if necessary, we can make $\bar\gamma$ an embedding.
The resulting map $\gamma$ may then also be taken to be an embedding by
using small enough normal discs.
\end{proof}

When considering manifolds with boundary, the preceding results suffice to do surgery
on the interior, leaving the boundary fixed. We can also do surgery on the boundary,
in the form of attaching handles to the boundary (or, what amounts to the same thing,
attaching traces of surgeries on the boundary). The inverse operation to the latter
is handle subtraction, which requires further comment.

Let $(M,\bndry M)$ be a $G$-manifold with nonempty boundary, let $(X,A)$ be a pair of $G$-spaces,
and let $(f,\xi,t)\colon (M,\bndry M)\to (X,A)$ be a normal map.
Represent an element of $\bar\pi_n(f)(x\colon G/H\to M)$
as the class of a quadruple $(\alpha,\beta,\gamma,\delta)$
as in the following diagram:
\[
 \xymatrix@!C=4em{
  & G/H\times S^{n-2} \ar[rr]^\alpha \ar[dl] \ar'[d][dd]
    & & \bndry M \ar[dd]^f \ar[dl] \\
  G/H\times D_-^{n-1} \ar[rr]^(.65)\gamma \ar[dd]
    & & M \ar[dd]^(.3)f \\
  & G/H\times D_+^{n-1} \ar'[r]_-\beta[rr] \ar[dl]
    & & A \ar[dl] \\
  G/H \times D^n \ar[rr]_\delta
    & & X
 }
\]
To say that {\em we can do handle subtraction on $(\alpha,\beta,\gamma,\delta)$}
is to say, first, that $\gamma$ is homotopic to a map that can be thickened to an embedding
\[
 \bar\gamma\colon G\times_H(D_-^{n-1}\times D(V)) \to M
\]
for some representation $V$ with $|M| = n + |V| + |G/H| - 1$,
such that the restriction of $\bar\gamma$ to $G\times_H(S^{n-2}\times D(V))$
is an embedding in $\bndry M$.
Because $\delta$ provides a homotopy of $f\gamma$ to a map into $A$, we can homotope
$f$ so that the image of $\bar\gamma$ lies entirely in $A$.
To do handle subtraction on $(\alpha,\beta,\gamma,\delta)$ then means
to remove the interior of the image of $\bar\gamma$ and smooth corners to get a new manifold
$(M_0,\bndry M_0)$ and a map $f_0\colon (M_0,\bndry M_0)\to (X,A)$,
with the trivialization $t_0 = t|M_0$ making $(f_0,\xi,t_0)$ again a normal map.

\begin{theorem}\label{thm:handlesubtraction}
Let $(M,\bndry M)$ be an ideal $G$-manifold, let $(X,A)$ be a pair of $G$-spaces,
and suppose that $(f,\xi,t)\colon (M,\bndry M)\to (X,A)$ is a normal map.
Let the quadruple $(\alpha,\beta,\gamma,\delta)$ represent an element of $\bar\pi_n(f)(x)$ as above.
Let $M^H_x$ denote the component of $M^H$ containing 
$x$ and let $W_x H$ denote the subgroup of $WH$
carrying $M^H_x$ to itself.
Let $m = |M^H_x|$ and let $w = |W_x H| = |WH|$.
Suppose that
\[
 n \leq m - w - 1
\]
and
\[
 n \leq m - w - |(M^H_x)^K|
\]
for every $K$ strictly containing $H$.
Then $\gamma$ is homotopic to a map that can be thickened to an immersion
\[
 G\times_H(D_-^{n-1}\times D^{m-w-n+1}\times D(W)) \to M
\]
that restricts to an immersion
\[
 G\times_H(S^{n-2}\times D^{m-w-n+1}\times D(W)) \to \bndry M
\]
for some $H$-representation $W$ with $W^H = 0$.
\end{theorem}

\begin{proof}
The proof is essentially the same as that of the first part of Theorem~\ref{thm:surgerypossible}.
We just have to observe that we can take a stable trivialization of $\alpha^*T(\bndry M)$
compatible with the stable trivialization of $\gamma^*TM$ given by the contractibility
of $D_-^{n-1}$, and use a relative version of the immersion results.
\end{proof}

\begin{corollary}\label{cor:handlesubtraction}
Under the assumptions of Theorem~\ref{thm:handlesubtraction},
if we also have
\[
 n \leq \lfloor (m-w+1)/2 \rfloor,
\]
then we can do surgery on $[\alpha,\beta,\gamma,\delta]$.
\qed
\end{corollary}

%%%%%%%%%%%%%%%%%%%%%%%%%%%%%%%%%%%%%%%%%%%%%%%%
\section{Surgery Below the Critical Dimension}

We are now in a position to describe the first, easier part of surgery,
making a given normal map connected up to a certain dimension.

\begin{definition}
We say that a $G$-map $f\colon X\to Y$ is a {\em $\Pi_0$-surjection} if,
for each $H\leq G$,
$f^H\colon X^H\to Y^H$
induces a surjection on sets of components.
\end{definition}

\begin{definition}\label{def:fixeddimensions}
Let $M$ be a $G$-manifold. For a point $x\in M$ with isotropy subgroup $H$,
let $M^H_x$ be the component of $M^H$ containing $x$, and let
\[
 m_x = |M^H_x|
\]
and
\[
 w_x = |WH|.
\]
The {\em critical dimension function} is then
\[
 \mu_x = \lfloor (m_x - w_x + 1)/2 \rfloor.
\]
\end{definition}

%\begin{definition}
%Let $f\colon (M,\bndry M)\to (X,Y)$ be a map from a $G$-manifold with (possibly empty) boundary to a pair of $G$-spaces.
%We say that $f$ is {\em connected below the critical dimension}
%if, first, for each closed subgroup $H$ of $G$, $f$ induces a one-to-one correspondence
%between the components of $M^H$ and $X^H$ and a one-to-one correspondence
%between the components of $\bndry M^H$ and $Y^H$.
%Assuming that true, for each $x\in M$ with isotropy $H$,
%let $X^H_x$ be the component of $X^H$ corresponding to $M^H_x$,
%let $\bndry M^H_x = M^H_x\intersect \bndry M$, and let $Y^H_x = X^H_x\intersect Y$
%(so $\bndry M^H_x$ and $Y^H_x$ need not be connected, nor even nonempty, but the
%restriction $\bndry M^H_x\to Y^H_x$ induces a one-to-one correspondence of components).
%We then require that
%\begin{enumerate}
%\item $f^H_x\colon M^H_x\to X^H_x$ is $\mu_x$-connected if $m_x-w_x$ is even,
%and $(\mu_x-1)$-connected if $m_x-w_x$ is odd;
%
%\item $\bndry f^H_x\colon \bndry M^H_x\to Y^H_x$ is $(\mu_x-1)$-connected
%on each component; and
%
%\item If $\bndry M^H_x \neq \emptyset$, then
%$(f^H_x,\bndry f^H_x)\colon (M^H_x,\bndry M^H_x) \to (X^H_x,Y^H_x)$
%is homologically $\mu_x$-connected.
%
%\end{enumerate}
%\end{definition}

\begin{definition}\label{def:fixedcomponents}
Let $(f,\bndry f)\colon (M,\bndry M)\to (X,A)$ be a map from a $G$-manifold with (possibly empty) boundary to a pair of $G$-spaces.
Suppose that both $f$ and $\bndry f$ are $\Pi_0$-equivalences.
For each $x\in M$ with isotropy $H$,
let $X^H_x$ be the component of $X^H$ corresponding to $M^H_x$,
let $\bndry M^H_x = M^H_x\intersect \bndry M$, and let $A^H_x = X^H_x\intersect A$
(so $\bndry M^H_x$ and $A^H_x$ need not be connected, nor even nonempty, but the
restriction $\bndry M^H_x\to A^H_x$ induces a one-to-one correspondence of components).
\end{definition}

We will want to apply Theorems~\ref{thm:surgerypossible} and~\ref{thm:handlesubtraction}
and their corollaries. In order to be sure that we satisfy the assumptions of those results,
we must assume a gap hypothesis, as is common and recognized to be necessary
for equivariant surgery.

\begin{definition}
Let $\tau$ be a representation of $\Pi_G X$ for a $G$-space $X$.
For $x\colon G/H\to X$, write $\tau(x) = G\times_H \tau_x$.
\begin{enumerate}
\item
We say that $\tau$ {\em satisfies the gap hypothesis}
if, for every $x\colon G/H\to X$ and every $K\leq H$, we have
either that $\tau_x^K = \tau_x^H$ or
\[
 |\tau_x^K| - |WK| \geq 2|\tau_x^H|.
\]

\item
We say that $\tau$ {\em has orbit spaces of dimension at least $n$} if
$|\tau_x^H| - |WH| \geq n$ for all injective maps $x\colon G/H\to X$.
\end{enumerate}
We say that a manifold $M$ satisfies the gap hypothesis or has orbit spaces of
dimension at least $n$ if the same is true of its tangent representation $\tau$.
\end{definition}

The following result shows what we can achieve by ``surgery below the critical dimension.''
It generalizes \cite[1.4]{Wal:surgery} and its proof will be similar.
The homological connectivity referred to in the conclusion is with respect
to the fixed-point homology defined in Definition~\ref{def:fixedsethomology}.
The assumption of $\Pi_0$-surjectivity will be satisfied automatically
in the case of a degree one map.
When $f$ satisfies the conclusion of the theorem, we say that it is
{\em connected up to the critical dimension}.

\begin{theorem}\label{thm:relativesurgerybelow}
Let $M$ be a compact $G$-manifold satisfying the gap hypothesis
and having orbit spaces of dimension at least three,
and let $(X,A)$ be a pair of finite $G$-CW complexes.
Let $(f,\xi,t)$ be a normal map with $(f,\bndry f)\colon (M,\bndry M)\to (X,A)$
such that both $f$ and $\bndry f$ are $\Pi_0$-surjections.
Then we can perform a finite sequence of equivariant surgeries on $M$
so that, after surgery, for each $x$ with isotropy $H$,
\begin{enumerate}
\item $f^H_x\colon M^H_x\to X^H_x$ is $\mu_x$-connected if $m_x-w_x$ is even,
or $(\mu_x-1)$-connected if $m_x-w_x$ is odd;

\item $\bndry f^H_x\colon \bndry M^H_x\to A^H_x$ is $(\mu_x-1)$-connected
on each component; and

\item If $\bndry M^H_x \neq \emptyset$, then
$(f^H_x,\bndry f^H_x)\colon (M^H_x,\bndry M^H_x) \to (X^H_x,A^H_x)$
is homologically $\mu_x$-connected.

\end{enumerate}
\end{theorem}

\begin{proof}
Throughout this proof, $\mu_x$ denotes the critical dimension function for $M$, not $\bndry M$.

We proceed by induction on the isotropy groups of $M$.
Fix an isotropy group $H$ of $M$ and suppose that
$(f^K,\bndry f^K)$ has the desired connectivity for every $K$ strictly containing $H$.
Fix an $x\colon G/H\to M$ with isotropy $H$.
Suppose that $1\leq n\leq \mu_x-1$ and, by induction on $n$, that
$\pi_k(\bndry f^H_x) = 0$ for $k < n$.
We have that $\pi_n(\bndry f^H_x)$ is a finitely generated $\Z\pi_0\Aut(x)$-module
via the relative Hurewicz isomorphism and the fact that $M$ and $X$ are finite complexes.
(When $n=1$, $\pi_1(\bndry f^H_x)$ is a finite union of $\pi_0\Aut(x)$-orbits;
when $n=2$, $\pi_2(\bndry f^H_x)$ is a finite product of
quotients of $\pi_0\Aut(x)$.)
Our assumptions on $M$ ensure that the assumptions of
Corollary~\ref{cor:surgerypossible} are satisfied for $\bndry M\to A$ so that we can do surgery
on a generator of $\pi_n(\bndry f^H_x)$.
Corollary~\ref{cor:surgeryKillsHomotopy} then applies to tell us
that the surgery kills that generator and hence reduces the number of generators.
Proceeding, we can kill all of the generators by a finite number of surgeries,
leaving $\pi_n(\bndry f^H_x) = 0$.
Moreover, these surgeries do not affect the fixed sets $\bndry M^K$ for $K$ strictly containing $H$,
so do not change the connectivity of $\bndry f^K$.
Nor do they introduce any new isotropy.
Proceeding by induction on $n$, we see that we can
make $\bndry f^H_x$ have the desired connectivity after a finite number of surgeries.
Stacking the traces of these surgeries together gives a normal bordism $N$ from the
original $\bndry M$ to a new manifold, say $P$. We attach $N$ to $M$ along $\bndry M$
so the resulting normal map to $(X,A)$ has the correct connectivity on the boundary
piece $P^H_x$.
To simplify notation, we call this manifold $(M,\bndry M)$ again.

By a similar argument, we can do surgery in the interior of $M$, leaving $\bndry M$
and $M^K$ unchanged for $K$ strictly containing $H$, to make
$f^H_x\colon M^H_x\to X^H_x$ be $\mu_x$-connected if $m_x-w_x$ is even, or
$(\mu_x-1)$-connected if $m_x-w_x$ is odd;
the difference comes in the step where we apply Corollary~\ref{cor:surgeryKillsHomotopy}.
Again, calling the new manifold by the old name, we may now assume that both
$f^H_x$ and $\bndry f^H_x$ have the desired connectivity.

If $m_x-w_x$ is even, the long exact sequence in homology
implies that $(f^H_x,\bndry f^H_x)$
is homologically $\mu_x$-connected as desired.
However, when $m_x-w_x$ is odd and $\bndry M^H_x$ is nonempty, we can conclude only that
$(f_x^H,\bndry f_x^H)$ is $(\mu_x-1)$-connected.
In this case we can do a further surgery (actually, a handle subtraction)
to make $(f_x^H,\bndry f_x^H)$ be $\mu_x$-connected, as we now explain.

Note that $H_{\mu_x}(f^H_x)\to H_{\mu_x}(f^H_x,\bndry f^H_x)$ is surjective
because the next group in the long exact sequence is 0.
Further, the Hurewicz map $\pi_{\mu_x}(f^H_x)\to H_{\mu_x}(f^H_x)$ is an isomorphism
and the latter group is a finitely generated $\Z\pi_0\Aut(x)$-module
because the spaces involved are finite complexes.
Take a finite generating set of $\pi_{\mu_x}(f^H_x)$
and apply Theorem~\ref{thm:surgerypossible} to represent
these generators as diagrams
\[
 \xymatrix{
	G\times_H (S^{\mu_x-1}\times D^{\mu_x}\times D(W)) \ar[r]^-{\bar\alpha} \ar[d]
	  & M \ar[d]^f \\
	G\times_H (D^{\mu_x}\times D^{\mu_x}\times D(W)) \ar[r]_-{\bar\beta}
	  & X
 }
\]
in which the maps $\bar\alpha$ are embeddings in the interior of $M$. 
Recalling that the argument looks at the projections
$S^{\mu_x-1}\to \breve M^H_x/W_xH$ and turns them into embeddings, 
we see that we can arrange that these
embeddings are pairwise disjoint.
Further, we can connect each sphere to a component of $\bndry \breve M^H_x/W_xH$ via a tube
(and keep these all disjoint). 
Going back up to $M^H_x$, we then have a disjoint collection of 
embedded, framed discs with boundaries
in $\bndry M^H_x$:
\[
 \xymatrix{
  W_xH\times S^{\mu_x-2}\times D^{\mu_x} \ar[r] \ar[d]
   & \bndry M^H_x \ar[d] \\
  W_xH\times D^{\mu_x-1}\times D^{\mu_x} \ar[r]
   & M^H_x.
 }
\]
The trivializations of the normal bundles in $M$ extend over the tubes to give us
a collection of embeddings of the following form:
\[
 \xymatrix{
  G\times_H (S^{\mu_x-2}\times D^{\mu_x}\times D(W)) \ar[r] \ar[d]
   & \bndry M \ar[d] \\
  G\times_H (D^{\mu_x-1}\times D^{\mu_x}\times D(W)) \ar[r]
   & M.
 }
\]
Let $Q$ denote the union of the images of these framed discs, or handles, in $M$ and
let $\bar M$ denote the result of removing these handles, i.e., removing
the interior of $Q$ from $M$ and then smoothing corners.
Using $\bar\beta$ together with the interiors of the added tubes, we can
homotope $f$ so that $f(Q)\subset A$, so that we get a normal map
\[
 (\bar f, \bndry\bar f)\colon (\bar M,\bndry\bar M)\to (X,A).
\]
We claim that $(\bar f^H_x, \bndry\bar f^H_x)$ has the required connectivity.

Define another map of pairs,
\[
 (f^H_x,g)\colon (M^H_x, Q^H_x\union\bndry M^H_x) \to (X^H_x, A^H_x).
\]
Consider the homology long exact sequence associated to the map
$(f^H_x, \bndry f^H_x) \to (f^H_x,g)$, where
we use the induced local coefficient system from $M^H_x$ throughout.
The third term is, with a shift of two in grading,
\[
 H_*(Q^H_x\union\bndry M^H_x,\bndry M^H_x) \iso H_*(Q^H_x, Q^H_x\intersect \bndry M^H_x)
\]
by excision.
Excision also allows us to identify
\[
 H_*(f^H_x,g) \iso H_*(\bar f^H_x, \bndry\bar f^H_x),
\]
so the long exact sequence takes the form
\begin{multline*}
\cdots \to
       H_{k-1}(Q^H_x, Q^H_x\intersect \bndry M^H_x) \to
       H_k(f^H_x, \bndry f^H_x) \to \\
       H_k(\bar f^H_x, \bndry\bar f^H_x) \to
       H_{k-2}(Q^H_x, Q^H_x\intersect \bndry M^H_x) \to \cdots
\end{multline*}
Now, the pair $(Q^H_x, Q^H_x\intersect \bndry M^H_x)$ is homotopy equivalent
to a disjoint union of pairs of the form $(W_xH\times D^{\mu_x-1}, W_xH\times S^{\mu_x-2})$,
so
$H_k(Q^H_x, Q^H_x\intersect \bndry M^H_x) = 0$ for $k<\mu_x-1$,
hence
\[
 H_k(\bar f^H_x, \bndry\bar f^H_x) \iso H_k(f^H_x, \bndry f^H_x) = 0
 \quad\text{for $k\leq \mu_x-1$.}
\]
At $\mu_x$ we thus have the following exact sequence:
\[
 H_{\mu_x-1}(Q^H_x, Q^H_x\intersect \bndry M^H_x)
  \to H_{\mu_x}(f^H_x, \bndry f^H_x)
  \to H_{\mu_x}(\bar f^H_x, \bndry\bar f^H_x)
  \to 0.
\]
The handles $Q^H_x$ were chosen to map to a set of generators of
$H_{\mu_x}(f^H_x, \bndry f^H_x)$, so the first map in this sequence is surjective,
hence $H_{\mu_x}(\bar f^H_x, \bndry\bar f^H_x) = 0$ as required.

On the boundary, we have performed a $(\mu_x-1)$ surgery on $\bndry M^H_x$,
so Corollary~\ref{cor:surgeryKillsHomotopy} tells us that $\bndry \bar f^H_x$ remains
$(\mu_x-1)$-connected. It follows from the long exact homology sequence that
$\bar f^H_x$ also remains $(\mu_x-1)$-connected.

Induction on $H$ then completes the proof.
\end{proof}

%%%%%%%%%%%%%%%%%%%%%%%%%%%%%%%%%%%%
\section{The Surgery Kernel}

We now suppose that we have a normal map that is connected up to the critical dimension
and examine the surgery kernel.

\begin{definition}
Let $(M,\bndry M)$ be a compact $G$-manifold,
let $(X,\bndry X)$ be a $G$-Poincar\'e duality space of dimension $\tau$,
and let $(f,\bndry f)\colon (M,\bndry M)\to (X,\bndry X)$ be
a degree-one map.
The {\em equivariant surgery kernels of $(f,\bndry f)$} are defined by
\begin{align*}
 K^G_*(f;\MackeyOp S) &= H^G_{*+1}(f;\MackeyOp S) \\
 K^G_*(f,\bndry f;\MackeyOp S) &= H^G_{*+1}(f,\bndry f;\MackeyOp S) \\
 \K^G_*(f;\Mackey T) &= \H^G_{*+1}(f;\Mackey T) \quad\text{ and}\\
 \K^G_*(f,\bndry f;\Mackey T) &= \H^G_{*+1}(f,\bndry f;\Mackey T)
\end{align*}
for any coefficient systems $\MackeyOp S$ and $\Mackey T$.
The {\em equivariant surgery cokernels of $(f,\bndry f)$} are defined similarly by
\begin{align*}
 K_G^*(f;\Mackey T) &= H_G^{*+1}(f;\Mackey T) \\
 K_G^*(f,\bndry f;\Mackey T) &= H_G^{*+1}(f,\bndry f;\Mackey T) \\
 \K_G^*(f;\MackeyOp S) &= \H_G^{*+1}(f;\MackeyOp S) \quad\text{ and}\\
 \K_G^*(f,\bndry f;\MackeyOp S) &= \H_G^{*+1}(f,\bndry f;\MackeyOp S).
\end{align*}
Note that the kernels and cokernels are graded on representations of $\Pi_G X$.
We shall write $\Mackey K^G_*(f)$ and so on for the $\stab\Pi_G X$-modules obtained
by taking universal coefficients, as in Definition~\ref{def:universalcoeffs}.

We also have the nonequivariant (integer-graded) kernels and cokernels
\begin{align*}
 K_*(f) &= H_{*+1}(f) \\
 K_*(f,\bndry f) &= H_{*+1}(f,\bndry f) \\
 K^*(f) &= H^{*+1}(f) \text{ and}\\
 K^*(f,\bndry f) &= H^{*+1}(f,\bndry f),
\end{align*}
which we will always take with local coefficients
$\Z\Pi X(x,-)$ for some $x\colon G/e\to X$, as in
Definition~\ref{def:fixedsethomology}.
\end{definition}

As nonequivariantly, and by the same argument, we have the following diagram
of split short exact sequences, in which the two right-most vertical arrows
are Poincar\'e duality isomorphisms, implying that the left-most vertical arrow
is also an isomorphism:
\[
 \xymatrix{
  0 & \K_G^{\tau-\alpha}(f;\MackeyOp S) \ar[l] \ar[d]
    & \H_G^{\tau-\alpha}(M;\MackeyOp S) \ar[l] \ar[d]_\iso
    & \H_G^{\tau-\alpha}(X;\MackeyOp S) \ar[l]_{f^*} \ar[d]_\iso
    & 0 \ar[l] \\
  0 \ar[r] & K^G_\alpha(f,\bndry f;\MackeyOp S) \ar[r]
    & H^G_\alpha(M,\bndry M;\MackeyOp S) \ar[r]_{f_*}
    & H^G_\alpha(X,\bndry X;\MackeyOp S) \ar[r]
    & 0
 }
\]
There are similar diagrams with the roles of $f$ and $(f,\bndry f)$ reversed, and with the
roles of $K$ and $\K$ reversed.
Thus, the kernels and cokernels exhibit Poincar\'e duality.

Suppose now that we have done surgery below the critical dimension,
so $(f,\bndry f)$ is connected up to the critical dimension.
Assuming that we are in a situation where we will be doing induction on isotropy groups,
consider a point $x\in M$ with isotropy $H$ and suppose that
$f^K$ and $\bndry f^K$ are homotopy equivalences for all $K$ strictly containing $H$.
We look at $(f^H_x,\bndry f^H_x)\colon (M^H_x,\bndry M^H_x)\to (X^H_x,\bndry X^H_x)$,
a $W_xH$-map that is a homotopy equivalence on all proper fixed sets.

To simplify notation, we recast the situation as follows.
Consider a compact Lie group $W$, a compact $W$-manifold $(M,\bndry M)$,
a $W$-Poincar\'e duality pair $(X,\bndry X)$, and a degree-one map
$(f,\bndry f)\colon (M,\bndry M)\to (X,\bndry X)$.
Let $m = |M|$, $w = |W|$, and $\mu = \lfloor (m-w+1)/2 \rfloor$.
We assume that $f^K$ and $\bndry f^K$ are homotopy equivalences for all nontrivial subgroups
and that
\begin{enumerate}
\item
$f$ is nonequivariantly $\mu$-connected if $m-w$ is even, and $(\mu-1)$-connected if $m-w$ is odd;
\item
$\bndry f$ is nonequivariantly $(\mu-1)$-connected; and
\item
if $\bndry M\neq \emptyset$, then $(f,\bndry f)$ is nonequivariantly homologically $\mu$-connected.
\end{enumerate}
Under these assumptions we want to examine the equivariant and nonequivariant surgery kernels.

\begin{theorem}\label{thm:equivariantkernel}
With the assumptions above we have
\begin{align*}
 K^W_n(f,\bndry f;\MackeyOp S) &= 0 \text{ and}\\
 K_W^n(f,\bndry f;\Mackey T) &= 0
\end{align*}
for all integers $n\neq \mu$ and all coefficient systems $\MackeyOp S$ and $\Mackey T$.
Further, $\Mackey K^W_\mu(f,\bndry f)$ is a finitely generated,
stably free $\stab\Pi_G X$-module with $W$-free stable basis.
\end{theorem}

\begin{proof}
By definition,
\[
 K^W_n(f,\bndry f;\MackeyOp S)
  = H^W_{n+1}(f,\bndry f;\MackeyOp S)
  \iso H^W_{n+1}(Cf,C\bndry f;\MackeyOp S)
\]
where $Cf$ denotes the mapping cone of $f$ over $X$.
By assumption, $(Cf)^K$ and $(C\bndry f)^K$ are contractible for all proper subgroups $K$
and $(Cf,C\bndry f)$ is nonequivariantly $(\mu+1)$-connected.
Therefore, we can approximate $(Cf,C\bndry f)$ by a relative $W$-CW complex $(Y,\bndry Y)$
over $X$ with only free relative cells of dimension $\mu+1$ or larger. It follows that
\[
 K^W_n(f,\bndry f;\MackeyOp S) = 0 \text{ for $n<\mu$.}
\]
Similarly, 
$Cf$ is nonequivariantly $(\mu+1)$-connected if $m-w$ is even and $\mu$-connected if $m-w$ is odd,
so we can approximate it by a based $W$-CW complex $Z$ over $X$ with only free cells,
of dimension at least $\mu+1$ if $m-w$ is even and at least $\mu$ if $m-w$ is odd.
This cell complex may also be viewed as a dual $W$-CW($\tau$) complex because the cells are free,
however, each cell $W/e\times D^n$ is considered to be a $(w+n)$-dimensional dual cell.
Using the fact that $|\tau| = m$, this gives us
\begin{multline*}
 \K_W^{\tau-n}(f;\MackeyOp S) = \tilde\H_W^{\tau-n+1}(Z;\MackeyOp S) \\
  = 0\text{ if }
  \begin{cases}
   m-n+1 < \mu+w+1 & \text{when $m-w$ is even} \\
   m-n+1 < \mu+w & \text{when $m-w$ is odd.}
  \end{cases}
\end{multline*}
However, the two cases both simplify to say that the cokernel vanishes when
$n > \mu$. As mentioned before the theorem, Poincar\'e duality gives us
\[
 K^W_n(f,\bndry f;\Mackey S) \iso \K_W^{\tau-n}(f;\MackeyOp S).
\]
The left group vanishes when $n < \mu$ while the right group vanishes when $n > \mu$.
Therefore, the only possible nonzero group occurs when $n = \mu$.
The same argument, {\em mutatis mutandi}, 
applies to cohomology, so the first statement of the theorem is shown.

The last statement of the theorem now follows from Lemma~\ref{lem:stablyfree}
applied to the integer-graded chain complex $\Mackey C^W_*(Y,\bndry Y)$,
using that the relative cells of $(Y,\bndry Y)$ are $W$-free.
\end{proof}

We then get the following consequence for the nonequivariant kernel.
This is a version of Petrie's result \cite[Theorem~3.4]{Pe:projectiveClassGroup}.

\begin{corollary}\label{cor:nonequivariantkernel}
With the assumptions of this section, we have
\[
 K_n(f,\bndry f) = 0 \text{ unless $\mu \leq n \leq \mu+w$},
\]
and in that range we have
\[
 K_{\mu+q}(f,\bndry f) \iso K^W_\mu(f,\bndry f; \MackeyOp H_q).
\]
Here, the coefficient system $\MackeyOp H_q$ is given by
\[
 \MackeyOp H_q(y\colon W/K\to X) = H_q(W/K;\Z\Pi X(x,-)),
\]
considering $W/K$ as a nonequivariant space over $X$.
Further, $K_\mu(f,\bndry f)$ is a finitely generated, stably free $\Z\pi_0\Aut(x)$-module.
\end{corollary}

\begin{proof}
Consider nonequivariant homology $H_*(-;\Z\Pi X(x,-))$ as an equivariant
homology theory on $W$-spaces over $X$ (by forgetting the $W$-action).
We then have an equivariant Atiyah-Hirzerbruch spectral sequence, which we apply to $(f,\bndry f)$, to get
\[
 E^2_{p,q} = K^W_p(f,\bndry f;\MackeyOp H_q) \convto
  K_{p+q}(f,\bndry f;\Z\Pi X(x,-)).
\]
The preceding theorem tells us that the $E^2$ term is concentrated at $p=\mu$,
hence the spectral sequence collapses and we have the isomorphism claimed in the corollary.
Because $W/K$ is a manifold of dimension no more than $w$, the coefficient system
$\MackeyOp H_q$ vanishes for $q < 0$ or $q>w$, hence we get the vanishing result
stated.

The last statement of the corollary follows from the last statement
of the preceding theorem and the calculation
\[
 \MackeyOp H_0 = H_0(-;\Z\Pi X(x,-)) \iso \stab\Pi_W X(x,-).
\]
To see this, let $y\colon W/K\to X$ and consider
\begin{align*}
 \stab\Pi_W X(x,y)
  &= [\susp^\infty_W x_+, \susp^\infty_W y_+]^W_X \\
  &= [W/e_+\smsh S, \susp^\infty_W G/K_+]^W_X \\
  &\iso [S, \susp^\infty G/K_+]_X.
\end{align*}
We can calculate this last group by taking the universal cover $\tilde X\to X$,
pulling back along $y$ to a covering space over $G/K$, and then taking 
the free group on $\pi_0$
of the total space. But this is the same as the 0th homology of the total space,
which is $H_0(G/K;\Z\Pi X(x,-))$.
Finally, we note that
\[
 \stab\Pi_WX(x,x) \iso \Z\pi_0\Aut(x).
\]
because $x\colon W/e\to X$,
so the fact that $\Mackey K^W_\mu(f,\bndry f)$ is a stably free $\stab\Pi_G X$-module
with $W$-free stable basis implies that
$K^W_\mu(f,\bndry f;\stab\Pi_W X(x,-))$ is a stably free $\Z\pi_0\Aut(x)$-module.
\end{proof}

%%%%%%%%%%%%%%%%%%%%%%%%%%%%%%%%%%%%
\section{The $\Pi$-$\Pi$ Theorem}

We can now prove a $\Pi$-$\Pi$ theorem following the argument given by Wall
in \cite{Wal:surgery}.

\begin{theorem}[$\Pi$-$\Pi$ Theorem]
Let $(X,\bndry X)$ be a $G$-Poincar\'e duality pair of dimension $\tau$. Suppose that $\tau$ is ideal, satisfies the gap hypothesis, and has fixed sets of dimension at least 6. Suppose further that $\Pi_G \bndry X \to \Pi_G X$ is an equivalence of groupoids over $\orb G$. If $M$ is a smooth compact $G$-manifold and $(f,\xi,t) \colon (M,\bndry M) \to (X, \bndry X)$ is a degree one normal map, then $(f,\xi,t)$ is normally cobordant to a $G$-homotopy equivalence.
\end{theorem}

\begin{proof}
We assume that we've done surgery below the critical dimension, 
as in Theorem~\ref{thm:relativesurgerybelow}.
We proceed by induction on fixed sets, so assume that $H$ is an isotropy subgroup of $M$ and that
$f^K$ and $\bndry f^K$ are homotopy equivalences for each $K > H$.
Let $x$ be a point of $M$ with isotropy $H$;
we use the notations introduced in Definitions~\ref{def:fixeddimensions}
and~\ref{def:fixedcomponents}.
We look at two separate cases, where $m_x-w_x$ is even and where it is odd.
%Take all $W_xH$-equivariant homology that follows with coefficients in $\stab\Pi_{W_xH}X^H_x(x,-)$.

\vskip 1ex
\noindent {\em Case $m_x-w_x$ even.}

We've assumed that we've done surgery below the critical dimension, so we have that
$f^H_x\colon M^H_x \to X^H_x$ is $\mu_x$-connected, $\bndry f^H_x$ is $(\mu_x-1)$-connected,
and $(f^H_x,\bndry f^H_x)$ is homologically $\mu_x$-connected.
It follows from Theorem~\ref{thm:equivariantkernel} that
$\Mackey K^{W_xH}_{n}(f^H_x,\bndry f^H_x) = 0$ for $n\neq\mu_x$ and that
$\Mackey K^{W_xH}_{\mu_x}(f^H_x,\bndry f^H_x)$ is 
a finitely generated, stably free $\stab\Pi_{W_xH}X^H_x$-module
with $W_xH$-free stable basis.

By Corollary~\ref{cor:surgerypossible}, 
we may do surgery on a trivial $(\mu_x-1)$-sphere in $\bndry M^H_x$,
which is to say, attach a handle to $M^H_x$ of the form
$W_xH \times D^{\mu_x}\times D^{m_x-w_x-\mu_x} = W_xH \times D^{\mu_x}\times D^{\mu_x}$ 
along an embedding
$W_xH\times S^{\mu_x-1}\times D^{\mu_x}\to \bndry M^X_H$;
this is (part of) the fixed set of a handle of the form
$G\times_H(D^{\mu_x}\times D^{\mu_x}\times D(W))$ where $W^H = 0$.
Because the attaching map is trivial, this has the effect in dimension $\mu_x$
of adding a free $\stab\Pi_{W_xH}X^H_x$-module, with $W_xH$-free basis, 
to $\Mackey K_{\mu_x}(f^H_x,\bndry f^H_x)$
(and it does not change the kernel below this dimension).
Thus, after attaching sufficiently many such handles, 
we may assume that $\Mackey K_{\mu_x}(f^H_x,\bndry f^H_x)$ is not just stably free,
but is actually a free
$\stab\Pi_{W_xH}X^H_x$-module with a finite, $W_xH$-free basis.

By Namioka \cite{Na:pairs} and Corollary~\ref{cor:nonequivariantkernel}, 
we have an isomorphism of $\Z\pi_0\Aut(x)$-modules
\[
 \pi_{\mu_x+1}(f^H_x,\bndry f^H_x) \iso K_{\mu_x}(f^H_x,\bndry f^H_x)
  \iso K^{W_xH}_{\mu_x}(f^H_x,\bndry f^H_x;\stab\Pi_{W_xH}X^H_x(x,-)).
\]
Using Theorem~\ref{thm:handlesubtraction}, 
we can therefore represent a basis of $\Mackey K^{W_xH}_{\mu_x}(f^H_x,\bndry f^H_x)$ by
immersions
\[
 \alpha_i\colon W_xH\times (D^{\mu_x}\times D^{\mu_x}, S^{\mu_x-1}\times D^{\mu_x})
  \to (M^H_x,\bndry M^H_x).
\]
Because these immersions land in the $W_xH$-free part of $M^H_x$, we can consider
the restrictions to $e\times D^{\mu_x}\times D^{\mu_x}$ and their
images in $M^H_x/W_xH$, and apply the ``piping'' argument in \cite[Ch.\ 4]{Wal:surgery}
verbatim to make those images homotopic to disjoint embeddings, then lift up the fibration,
so that we can make the $\alpha_i$ be a collection of disjoint embeddings.

We then perform handle subtraction using Corollary~\ref{cor:handlesubtraction}.
Let $U \homeo \coprod_i W_xH\times D^{\mu_x}\times D^{\mu_x}$ be the union of the image of
the $\alpha_i$ and let $M_0$ be result of removing the interior of $U$ from $M_x^H$.
Then $\Mackey H^{W_xH}_{\mu_x}(U,U\intersect \bndry M_x^H)$ is a free 
$\stab\Pi_{W_xH}X^H_x$-module with $W_xH$-free basis, by the dimension
axiom, and
\[
 \Mackey H^{W_xH}_n(U,U\intersect \bndry M_x^H)
  \iso \Mackey H^{W_xH}_n(U\union\bndry M_x^H, \bndry M_x^H)
  \to \Mackey K^{W_xH}_n(f_x^H,\bndry f_x^H)
\]
is an isomorphism for all $n$ by the choice of the $\alpha_i$ and the fact that these groups are
nonzero only in dimension $\mu_x$.
We also have the following diagram of long exact sequences:
\[
 \xymatrix{
   \Mackey H^{W_xH}_n(U\union\bndry M_x^H, \bndry M_x^H) \ar[r] \ar[d]_\iso
    & \Mackey H^{W_xH}_n(M_x^H,\bndry M_x^H) \ar[r] \ar@{=}[d]
    & \Mackey H^{W_xH}_n(M_x^H, U\union\bndry M_x^H) \ar[d] \\
   \Mackey K^{W_xH}_n(f_x^H,\bndry f_x^H) \ar@{>->}[r]
    & \Mackey H^{W_xH}_n(M_x^H,\bndry M_x^H) \ar@{->>}[r]
    & \Mackey H^{W_xH}_n(X_x^H, \bndry X_x^H)
 }
\]
This implies that
\[
 \Mackey H^{W_xH}_*(M_0,\bndry M_0) \iso \Mackey H^{W_xH}_*(M_x^H, U\union\bndry M_x^H)
  \iso \Mackey H^{W_xH}_*(X_x^H, \bndry X_x^H).
\]
Duality then gives $\Mackey \H_{W_xH}^*(X_x^H) \iso \Mackey \H_{W_xH}^*(M_0)$, which, 
by Corollary~\ref{cor:universalisomorphism},
implies
$\Mackey H^{W_xH}_*(M_0) \iso \Mackey H^{W_xH}_*(X_x^H)$.
The long exact sequence in homology then gives 
$\Mackey H^{W_xH}_*(\bndry M_0) \iso \Mackey H^{W_xH}_*(\bndry X_x^H)$.
This completes the case when $m_x-w_x$ is even.

\vskip 1ex
\noindent {\em Case $m_x-w_x$ odd.}

In this case, we have $\Mackey K^{W_xH}_n(f_x^H,\bndry f_x^H) = 0$ unless $n = \mu_x$, and
$\Mackey K^{W_xH}_n(f_x^H) = 0$ unless $n = \mu_x-1$. Hence, we have a short exact sequence
\[
 0 \to \Mackey K^{W_xH}_{\mu_x}(f_x^H,\bndry f_x^H) \to \Mackey K^{W_xH}_{\mu_x-1}(\bndry f_x^H)
  \to \Mackey K^{W_xH}_{\mu_x-1}(f_x^H) \to 0.
\]
As before, we can perform trivial surgeries on $\bndry M$ to convert these stably free
modules to free modules with $W_xH$-free bases. 
We then represent a basis of $\Mackey K^{W_xH}_{\mu_x}(f_x^H,\bndry f_x^H)$
by a collection of immersions 
\[
 \alpha_i\colon W_xH \times (D^{\mu_x}\times D^{\mu_x-1},S^{\mu_x-1}\times D^{\mu_x-1})
   \to (M_x^H, \bndry M_x^H)
\]
that extend to immersions $\alpha_i\times D(W)$ in $(M,\bndry M)$.
By looking at the images in the $W_xH$ orbit space, using the argument in \cite[Ch.\ 4]{Wal:surgery},
we can modify the $\alpha_i$ so that the restrictions to bounding spheres give disjoint
embeddings in $\bndry M_x^H$.

Now attach handles to $\bndry M_x^H$ using the restrictions 
$\alpha_i|W_xH\times S^{\mu_x-1}\times D^{\mu_x-1}\times D(W)$,
let $(\bar M, \bndry \bar M)$ be the resulting manifold.
Let $\bar f\colon \bar M^H_x\to X_x^H$ be the resulting map, and
let $U$ be the union of the attached handles.

Because the attaching maps are null-homotopic in $M_x^H$, we have
\[
 \Mackey K^{W_xH}_{\mu_x-1}(\bar M^H_x) \iso \Mackey K^{W_xH}_{\mu_x-1}(M_x^H)
\]
and $\Mackey K^{W_xH}_{\mu_x}(\bar M^H_x)$ is free, with a $W_xH$-basis given by the images of the thickened balls
consisting of the original $\alpha_i$ with the handles attached.

The triple
$\bndry \bar M^H_x\to \bndry \bar M^H_x\union U \to \bar M^H_x$ gives the following exact sequence,
using excision:
\begin{multline*}
 0 \to \Mackey K^{W_xH}_{\mu_x}(\bar f^H_x,\bndry \bar f^H_x) 
  \to \Mackey K^{W_xH}_{\mu_x}(f_x^H,\bndry f_x^H) \\
  \to \Mackey H^{W_xH}_{\mu_x-1}(U, U\intersect \bndry \bar M^H_x)
  \to \Mackey K^{W_xH}_{\mu_x-1}(\bar f^H_x,\bndry \bar f^H_x) \to 0.
\end{multline*}
The dual of the connecting homomorphism in the middle of this exact sequence is the connecting map
\[
 \Mackey \K_{W_xH}^{w_x+\mu_x-1}(f_x^H) 
  \to \Mackey \H_{W_xH}^{w_x+\mu_x}(\bar M^H_x,M_x^H)
  \iso \Mackey \H_{W_xH}^{w_x+\mu_x}(U,U\intersect \bndry M_x^H).
\]
This map is zero because the attaching maps are trivial in $M_x^H$,
which makes the next map in the long exact sequence a split inclusion.
Therefore, the long exact sequence above splits into two isomorphisms,
\[
 \Mackey K^{W_xH}_{\mu_x}(\bar f^H_x,\bndry \bar f^H_x) 
  \iso \Mackey K^{W_xH}_{\mu_x}(f_x^H,\bndry f_x^H)
\]
and
\[
 \Mackey H^{W_xH}_{\mu_x-1}(U, U\intersect \bndry \bar M^H_x)
  \iso \Mackey K^{W_xH}_{\mu_x-1}(\bar f^H_x,\bndry \bar f^H_x).
\]
By construction, the map
\[
 \Mackey K^{W_xH}_{\mu_x}(\bar f^H_x) 
  \to \Mackey K^{W_xH}_{\mu_x}(\bar f^H_x,\bndry \bar f^H_x) 
  \iso \Mackey K^{W_xH}_{\mu_x}(f_x^H,\bndry f_x^H)
\]
is an isomorphism. Therefore, the dual map
\[
 \Mackey \K_{W_xH}^{w_x+\mu_x-1}(\bar f^H_x,\bndry \bar f^H_x) 
 \xrightarrow{\iso} \Mackey \K_{W_xH}^{w_x+\mu_x-1}(\bar f^H_x)
\]
is also an isomorphism.
Because these are the first nonzero degrees in both dual homology and cohomology, it follows that
the map
\[
 \MackeyOp \K^{W_xH}_{w_x+\mu_x-1}(\bar f^H_x) 
  \xrightarrow{\iso} \MackeyOp \K^{W_xH}_{w_x+\mu_x-1}(\bar f^H_x,\bndry \bar f^H_x)
\]
is also an isomorphism.
Because the cells are all free, we can reinterpret dual cells as ordinary cells
with a shift by $w_x$, and it then follows that
\[
 \Mackey K^{W_xH}_{\mu_x-1}(\bar f^H_x) 
  \xrightarrow{\iso} \Mackey K^{W_xH}_{\mu_x-1}(\bar f^H_x,\bndry \bar f^H_x)
\]
is an isomorphism, which completes showing that these modules are isomorphic
in all degrees.

From the long exact sequence, we have that $\Mackey K^{W_xH}_*(\bndry \bar f^H_x) = 0$, 
hence $\bndry \bar M^H_x\to \bndry X^H_x$
is an equivalence.

By the equivariant Atiyah-Hirzebruch spectral sequence,
the nonequivariant kernel $K_{\mu_x-1}(\bar f^H_x)$ is isomorphic to 
$\Mackey K^{W_xH}_{\mu_x-1}(\bar f^H_x)(x)$,
so is a free $\Z\pi_0\Aut(x)$-module. 
Choose a basis and then, using the Hurewicz-Namioka isomorphism,
represent the basis elements as elements of $\pi_{\mu_x}(\bar f^H_x)$,
and use them to do surgery away from the boundary,
getting a new manifold $\hat f\colon \hat M\to X$ with $\bndry \hat M = \bndry \bar M$.
Let $P$ be the trace of the surgery and identify $\bar M\union \bndry \bar M\times I$ with $\bar M$,
so that $\bndry P =  \bar M \union_{\bndry \bar M} \hat M$. We then have a map
\[
 F\colon (P; \bar M, \hat M) \to (X\times I; X\times 0\union \bndry X\times I, X\time 1).
\]
Considering the pair $(P,\bar M)$, for which 
$\Mackey H^{W_xH}_*(P^H_x,\bar M^H_x)$ is concentrated in degree $\mu_x$,
we get the following exact sequence:
\begin{multline*}
 0 \to \Mackey K^{W_xH}_{\mu_x}(\bar f^H_x) 
   \to \Mackey K^{W_xH}_{\mu_x}(F^H_x) \\
   \to \Mackey H^{W_xH}_{\mu_x}(P^H_x,\bar M^H_x)
   \xrightarrow{\bndry_1} \Mackey K^{W_xH}_{\mu_x-1}(\bar f^H_x) 
   \to \Mackey K^{W_xH}_{\mu_x-1}(F^H_x) \to 0
\end{multline*}
Further, $\bndry_1$ is an isomorphism by construction, so 
$\Mackey K^{W_xH}_{\mu_x-1}(F^H_x) = 0$
and 
\[
 \Mackey K^{W_xH}_{\mu_x}(\bar f^H_x) \iso \Mackey K^{W_xH}_{\mu_x}(F^H_x).
\]
Now consider the pair $(P,\hat M)$, which gives the following exact sequence:
\[
 0 \to \Mackey K^{W_xH}_{\mu_x}(\hat f^H_x) 
   \to \Mackey K^{W_xH}_{\mu_x}(F^H_x) 
   \xrightarrow{i_*} \Mackey H^{W_xH}_{\mu_x}(P^H_x,\hat M^H_x)
   \to \Mackey K^{W_xH}_{\mu_x-1}(\hat f^H_x) \to 0
\]
By a similar argument to the one we made earlier, $i_*$ is dual to $\bndry_1$,
hence is an isomorphism. From this we see that
\[
 \Mackey K^{W_xH}_{\mu_x}(\hat f^H_x) = 0 = \Mackey K^{W_xH}_{\mu_x-1}(\hat f^H_x),
\]
so $\Mackey K^{W_xH}_*(\hat f^H_x) = 0$ in all degrees, i.e., $\hat M^H_x\to X^H_x$ is an equivalence.
Because we know that $\bndry \hat M^H_x\to \bndry X^H_x$ is an equivalence, it follows from the long
exact sequence that we have $\Mackey K^{W_xH}_*(\hat f^H_x,\bndry \hat f^H_x) = 0$ as well.

This completes the case in which $m_x-w_x$ is odd, 
and the full result is completed by induction on $H$.
\end{proof}